\documentclass[12pt]{amsart}

\usepackage[utf8]{inputenc}						
\usepackage{amsmath,amsfonts,amssymb,amsthm}	

\usepackage{geometry}		
\usepackage{mathtools}		
\usepackage{hyperref}		
\usepackage{setspace}		
\usepackage{cite}			
\usepackage{color}			

\definecolor{midpurple}{rgb}{0.6,0.2,0.4}
\hypersetup{colorlinks=true,citecolor=blue,linkcolor=blue,urlcolor=midpurple}

\numberwithin{equation}{section}			
\newtheorem{theorem}{Theorem}[section]			
\newtheorem{proposition}[theorem]{Proposition}	
\newtheorem{lemma}[theorem]{Lemma}
\newtheorem{corollary}[theorem]{Corollary}

\newtheorem{remark}[theorem]{Remark}

\newtheorem*{theorem*}{Theorem}
\newtheorem*{proposition*}{Proposition}
\newtheorem*{lemma*}{Lemma}
\newtheorem*{corollary*}{Corollary}
\newtheorem*{fact*}{Fact}

\newtheorem*{conjecture*}{Conjecture}
\newtheorem*{question*}{Question}
\newtheorem*{remark*}{Remark}
\newtheorem*{definition*}{Definition}

\renewcommand{\leq}{\leqslant}	
\renewcommand{\geq}{\geqslant}	

\newcommand{\N}{\mathbb{N}}		
\newcommand{\Z}{\mathbb{Z}}		
\newcommand{\R}{\mathbb{R}}		
\newcommand{\C}{\mathbb{C}}		
\newcommand{\T}{\mathbb{T}}		

\newcommand{\eps}{\varepsilon}			

\newcommand{\wt}{\widetilde}		
\newcommand{\wh}{\widehat}			
\DeclareMathOperator{\Supp}{Supp}	

\newcommand{\dx}{\mathrm{d}x}
\newcommand{\dm}{\mathrm{d}m}

\newcommand{\dalpha}{\mathrm{d}\alpha}
\newcommand{\deta}{\mathrm{d}\eta}
\newcommand{\dlambda}{\mathrm{d}\lambda}

\newcommand{\dtheta}{\mathrm{d}\theta}
\newcommand{\dxi}{\mathrm{d}\xi}
\newcommand{\dSigma}{\mathrm{d}\Sigma}

\newcommand{\bfa}{\mathbf{a}}

\newcommand{\bfk}{\mathbf{k}}

\newcommand{\bfn}{\mathbf{n}}

\newcommand{\bfu}{\mathbf{u}}

\newcommand{\bfx}{\mathbf{x}}

\newcommand{\bfP}{\mathbf{P}}

\newcommand{\bfalpha}{\boldsymbol{\alpha}}	
	
\newcommand{\bfell}{\boldsymbol{\ell}}

\newcommand{\bftheta}{\boldsymbol{\theta}}	

\newcommand{\bfxi}{\boldsymbol{\xi}}	

\newcommand{\dbfx}{\mathrm{d}\mathbf{x}}
\newcommand{\dbfalpha}{\mathrm{d}\boldsymbol{\alpha}}
	
\newcommand{\dbftheta}{\mathrm{d}\boldsymbol{\theta}}
\newcommand{\dbfxi}{\mathrm{d}\boldsymbol{\xi}}

\newcommand{\frakM}{\mathfrak{M}}		
\newcommand{\frakm}{\mathfrak{m}}		
\newcommand{\frakS}{\mathfrak{S}}		
\newcommand{\frakJ}{\mathfrak{J}}		
\newcommand{\calU}{\mathcal{U}}			
\newcommand{\calP}{\mathcal{P}}			

\begin{document}

\title{Restriction estimates of $\eps$-removal type
for $k$-th powers and paraboloids}

\author{Kevin Henriot, Kevin Hughes}

\date{}

\begin{abstract}
We obtain restriction estimates of $\eps$-removal type for 
the set of $k$-th powers of integers, and for 
discrete $d$-dimensional surfaces of the form
\begin{align*}
\{ (n_1,\dots,n_d,n_1^k + \dotsb + n_d^k) \,:\, |n_1|,\dots,|n_d| \leq N \},
\end{align*}
which we term '$k$-paraboloids'.
For these surfaces, we obtain a satisfying range of
exponents for large values of $d,k$.
We also obtain estimates of $\eps$-removal type
in the full supercritical range
for $k$-th powers and for $k$-paraboloids of dimension $d < k(k-2)$.
We rely on a variety of techniques in discrete harmonic analysis originating 
in Bourgain's works on the restriction theory 
of the squares and the discrete parabola.
\end{abstract}

\maketitle

\section{Introduction}
\label{sec:intro}

We are interested in restriction theorems
for discrete surfaces in $\Z^d$.
We restrict our attention to parametric surfaces
of the form
\begin{align}
\label{eq:intro:ParamSurface}
	S = \{\, \bfP(\bfn) \,:\, \bfn \in [-N,N]^d \,\}
\end{align}
where $\bfP = (P_1,\dots,P_r)$ is a system of
$r$ integer polynomials in $d$ variables,
and we assume that the map $\bfP : \Z^d \rightarrow \Z^r$ is injective for simplicity.
When the polynomials $P_1,\dots,P_r$ have degree $k_1,\dots,k_r$,
we define the total degree of the system $\bfP$
as $K = k_1 + \dotsb + k_r$.
We denote the action of the extension operator on
a sequence $a : \Z^d \rightarrow \C$ supported on $[-N,N]^d$ by
\begin{align*}
&\phantom{(\bfalpha \in \T^r)} &
F_a^{(\bfP)}(\bfalpha)
&= 
\sum_{\bfn \in \Z^d} a(\bfn) e\big( \bfP(\bfn) \cdot \bfalpha \big)
&&(\bfalpha \in \T^r).
\end{align*}
The natural restriction conjecture, 
based on heuristics from the circle method, is that
the $\eps$-free estimate
\begin{align}
\label{eq:intro:RestrEpsFree}
	\| F_a^{(\bfP)} \|_p^p \lesssim N^{\frac{dp}{2} - K} \| a \|_2^p
\end{align}
holds in the supercritical range $p > \tfrac{2K}{d}$,
the $\eps$-full estimate
\begin{align}
\label{eq:intro:RestrEpsFull}
\| F_a^{(\bfP)} \|_q^q 
\lesssim_\eps
N^{\frac{dq}{2} - K + \eps} \| a \|_2^q
\end{align}
holds at the critical exponent $q = \frac{2K}{d}$,
and the subcritical estimate
\begin{align}
\label{eq:intro:RestrSubcritical}
\| F_a^{(\bfP)} \|_r^r 
\lesssim_\eps
N^{\eps} \| a \|_2^r
\end{align}
holds for $2 \leq r < \frac{2K}{d}$.
This conjecture has to be corrected
when the discrete surface $\{ \bfP(\bfx),\,\bfx \in \Z^d \}$ contains
large special subvarieties,
but this does not appear to be the case for the surfaces we study. 

In the supercritical range, Bourgain resolved the natural restriction conjecture 
in the case $\bfP = (x^2)$ of the squares~\cite{Bourgain:Squares}
and in the case $\bfP = (x,x^2)$ of the $2D$ parabola~\cite{Bourgain:ParabI},
via discrete versions of the Tomas--Stein argument~\cite[Chapter~7]{Wolff:Book}
and the Hardy--Littlewood circle method.
Bourgain and Demeter~\cite{BD:DecouplConj}
later established the $\eps$-full estimate~\eqref{eq:intro:RestrEpsFull} 
for arbitrary definite irrational paraboloids $\bfP = (x_1, \dots, x_d, \theta_1 x_1^2 + \dotsb + \theta_d x_d^2)$
with $\theta_i \in (0,1]$ in the full supercritical range $p \geq \frac{2(d+2)}{d}$,
by developing powerful methods of multilinear harmonic analysis
(the indefinite case was later resolved in~\cite{BD:Hypersurf}).
In the rational case $\theta_1 = \dots = \theta _d = 1$,
the $\eps$-loss can be eliminated via Bourgain's earlier work~\cite{Bourgain:ParabI}.
In an important recent work, Killip and Vi\c{s}an~\cite{KV:Parab}
removed the $\eps$-loss for all definite parabolas,
using new techniques partly inspired by Bourgain's~\cite{Bourgain:ParabI}.
This note relies only on the earlier number-theoretic approach
of Bourgain~\cite{Bourgain:Squares,Bourgain:ParabI},
albeit with significant modifications,
since it is more adapted to our objective.
Indeed, we primarily seek to obtain weaker estimates of the form
\begin{align}
\label{eq:intro:TruncRestrEst}
	\int_{|F_a^{(\bfP)}| \geq N^{d/2 - \zeta} \|a\|_2} |F_a^{(\bfP)}|^q \dm
	\lesssim 
	N^{\frac{dq}{2} - K} \| a \|_2^q,
\end{align}
for a certain $\zeta > 0$,
in the complete supercritical range of exponents $q > \frac{2K}{d}$,
or a good approximation thereof.
We succeed in doing so for several classes of surfaces
generalizing that of the squares and the parabola.

Before introducing these results,
we discuss our motivation to seek $\eps$-removal estimates of the form~\eqref{eq:intro:TruncRestrEst}.
Justifying their terminology, these estimates can be used to remove the extraneous factor $N^\eps$
in~\eqref{eq:intro:RestrEpsFull},
as recalled in Lemma~\ref{thm:prelims:epsremoval} below.
Methods of multilinear harmonic analysis~\cite{BD:DecouplConj},
or even moment bounds exploiting arithmetic information
typically produce a factor of this form.
While the $N^\eps$ factor is sometimes inconsequential,
the sharp estimate~\eqref{eq:intro:RestrEpsFree} 
is often necessary in applications to additive combinatorics. 
More specifically, restriction estimates
are of key importance in the study of linear equations 
of the form $\sum_{i=1}^s \lambda_i \bfP(\bfn_i) = 0$,
where the $\lambda_i$ are non-zero integer coefficients summing to zero 
and the variables $\bfn_i$ lie in a sparse subset of $\{ 1,\dots, N \}^d$
(or in a sparse subset of $( \calP \cap \{ 1,\dots, N \} )^d$,
where $\calP$ are the prime numbers).
When the system of polynomials $\bfP$ is translation-invariant\footnote{
That is, when $\bfP(x_1 + t,\dots,x_d + t) = \bfP(x_1,\dots,x_d)$ for all $x_1,\dots,x_d,t \in \R$.},
this system of equations can be studied
via density-increment-based strategies~\cite{Smith:DiophRothI,Keil:DiophRoth,Henriot:logkeil,Henriot:addeqs}
exploiting $L^\infty \rightarrow L^p$ or $L^2 \rightarrow L^p$
restriction estimates for the surface~\eqref{eq:intro:ParamSurface};
we refer to~\cite{Henriot:addeqs} for a more complete discussion.
In the general case, one can also analyze such systems by
transference-based strategies~\cite{Green:RothPrimes,BP:RothSquares,Chow:RothPrimePowers}
which rely only on $L^\infty \rightarrow L^p$ estimates, although these take
a more complicated shape due to the presence of the $W$-trick.

Note also that truncated estimates of the form~\eqref{eq:intro:TruncRestrEst}
can be completed into full estimates of the form~\eqref{eq:intro:RestrEpsFree}
for a large enough range of exponents,
whenever a subcritical estimate of the form~\eqref{eq:intro:RestrSubcritical} is known
(which is always the case for $r = 2$).
This familiar procedure is recalled in Lemma~\ref{thm:prelims:SubcriticalCompletion} below,
but it generally gives a poor range of exponents due to the
smallness of the parameter~$\zeta$,
which is related to Weyl exponents.

The first surface we study is
\begin{align}
\label{eq:intro:PowersSurface}
	S = \{\, n^k \,:\, n \in \{1,\dots,N\} \,\},
\end{align}
corresponding to the system of polynomials $\bfP = (x^k)$ 
of total degree $k$, when $k \geq 3$ is an integer.
In this case we obtain the complete supercritical range of exponents
for epsilon-removal 
and a restricted range of exponents for truncated restriction estimates.

\begin{theorem}
\label{thm:intro:TruncRestrPowers}
Let $k \geq 3$
and $\tau = \max(2^{1-k},\frac{1}{k(k-1)})$,
and write $\bfP = (x^k)$.
The estimate~\eqref{eq:intro:TruncRestrEst} holds for any $p > 2k$
and $\zeta < \tau/2$,
and the estimate~\eqref{eq:intro:RestrEpsFree} holds 
for $p > 2 + 2 (k - 1) / \tau$.
\end{theorem}

The proof of this result consists in an adaptation
of Bourgain's argument for squares~\cite{Bourgain:Squares}.
We comment in Section~\ref{sec:exten} on the results that can be obtained for arbitrary
monomial curves by this approach.
It turns out that one only obtains the whole supercritical range
for the curve $(n^k)$, due to the lack of
efficient majorants of Weyl sums on major arcs in other cases.

Let $R_{s,k}(n)$ denote the number of representations of $n$ as a sum of $s$ $k$-th powers of integers. 
Hypothesis $K$ of Hardy and Littlewood~\cite[Section~17]{VW:WaringSurvey}
states that $R_{k,k}(n) \lesssim_\eps n^{\eps}$ for $k \geq 2$. 
It is known (and easy to show) for $k = 2$, and while it has been disproved for $k=3$ 
by Mahler~\cite{Mahler:HypothesisK}, it remains open for $k \geq 4$. 
Under this strong hypothesis, which is far out of reach of current methods, 
our epsilon-removal estimate implies the full supercritical range 
of conjectured restriction estimates for $k$-th powers.

\begin{corollary}
\label{thm:intro:RestrPowersHypK}
Let $k \geq 3$ and write $\bfP = (x^k)$.
If Hypothesis K is true,
the estimate~\eqref{eq:intro:RestrEpsFree} holds for $p > 2k$.
\end{corollary}

Fix a dimension $d \geq 1$ and a degree $k \geq 3$.
The next surface we study is the truncated $d$-dimensional $k$-paraboloid 
embedded in $\Z^{d+1}$
\begin{align}
\label{eq:intro:ParabSurface}
	S = \{\, (n_1,\dots,n_d,n_1^k + \dotsb + n_d^k) \,,\, n_i \in [-N,N] \cap \Z \,\},
\end{align}
which is the usual paraboloid when $k = 2$.
Our first theorem simplifies the approach of Bourgain for the parabola~\cite{Bourgain:ParabI}; 
the cost of our simplification is that we do not acquire the full supercritical range, and in particular, 
we ``lose $k$ variables'' from the conjectured range.

\begin{theorem}
\label{thm:intro:TruncRestrParabHighDim}
Suppose that $d \geq 1$ and $k \geq 3$,
and let $\tau = \max(2^{1-k},\frac{1}{k(k-1)})$.
Write also $\bfP = (x_1,\dots,x_d,x_1^k + \dotsb + x_d^k)$.
The truncated estimate~\eqref{eq:intro:TruncRestrEst}
holds for $\zeta < \frac{d\tau}{2}$ and 
$p > \frac{2(d+k) + 2k}{d}$,
and the estimate~\eqref{eq:intro:RestrEpsFree}
holds for $p > 2 + \frac{2k}{d\tau}$.
\end{theorem} 

Note that the exponent $\frac{2(d+k)}{d} + \frac{2k}{d}$
approximates the critical exponent when the dimension $d+1$ of the ambient space is large 
with respect to the degree $k$ of the paraboloid.
In our proof, this reflects the fact that 
the splitting behavior~\eqref{eq:parabhigh:Fsplitting} of exponential sums
dominates for large dimensions.
By adapting the difficult argument of Bourgain~\cite{Bourgain:ParabI} in a more direct fashion,
we can recover the complete supercritical range of exponents,
but only for sufficiently small dimensions, of size roughly less than the square of the degree.

\begin{theorem}
\label{thm:intro:TruncRestrParabLowDim}
Suppose that $d \geq 1$ and $k \geq 3$,
and let $\tau = \max(2^{1-k},\frac{1}{k(k-1)})$.
Write also $\bfP = (x_1,\dots,x_d,x_1^k + \dotsb + x_d^k)$,
and assume that $d < \frac{k^2 - 2k}{1 - k\tau}$.
Then the estimate~\eqref{eq:intro:TruncRestrEst}
holds for any $\zeta < \frac{d\tau}{2}$
and $p > \frac{2(d+k)}{d}$.
\end{theorem}

The proof of Theorems~\ref{thm:intro:TruncRestrParabHighDim}
and~\ref{thm:intro:TruncRestrParabLowDim} exploits
available bounds on one-dimensional Weyl sums of degree $k$,
such as estimates of Weyl type and the Poisson formula on major arcs.
The poor quality of known minor arc bounds
is the main reason for our relative condition on~$d$ and~$k$
in Theorem~\ref{thm:intro:TruncRestrParabLowDim}.
It is a curious feature that in dimension $d = 1$ (say),
the approach of Bourgain~\cite{Bourgain:ParabI}
apparently yields the whole supercritical range for the ``sparse''
curve $(x,x^k)$. 
Note that this removes the $\eps$-loss 
in the restriction estimates of Hu and Li~\cite{HuLi:Degree3,HuLi:Degreed} for these curves.
We remark also that for very large dimensions,
the method of proof of Theorem~\ref{thm:intro:TruncRestrParabLowDim}
also yields restriction exponents, but the range so obtained
is much narrower than that of Theorem~\ref{thm:intro:TruncRestrParabHighDim}.

We make a last remark about the exponent
$\tau$ in Theorems~\ref{thm:intro:TruncRestrParabHighDim}
and Theorem~\ref{thm:intro:TruncRestrParabLowDim},
which affects dramatically the quality of full restriction
estimates we can obtain as corollaries.
In those results, one can in fact take 
$\tau$ to be the largest exponent such that,
for all $\alpha \in \T$ such that there exists $q,a \in \Z$
with $N \leq q \leq N^k$ and 
$\| \alpha - \frac{a}{q} \| \leq \frac{1}{qN^{k-1}}$, 
\begin{equation}
\label{eq:intro:BestWeylEst}
	\bigg| \sum_{n=1}^N e(\alpha n^k + n\theta) \bigg|
	\lesssim_\eps N^{1 + \eps} \Big( \frac{1}{q} + \frac{1}{N} + \frac{q}{N^k} \Big)^\tau 
\end{equation}
uniformly in $\theta \in \T$.
For a fixed degree $k$, the best one can hope for 
$\tau$ is to be $1/k$; see e.g.~\cite[Problem~8,~p.~196]{Montgomery:Book}. 
If~\eqref{eq:intro:BestWeylEst} were to hold for all $\tau < 1/k$, 
then Theorem~\ref{thm:intro:TruncRestrParabLowDim} would 
improve to the full supercritical range in all dimensions. 
Instead the range in Theorem~\ref{thm:intro:TruncRestrParabLowDim} relies 
on the best known unconditional exponent $\tau$,
which is $\tau = \frac{1}{k(k-1)}$ for large values of $k$
by Bourgain--Demeter--Guth's recent resolution of Vinogradov's mean value 
conjecture~\cite{BDG:VinoMeanValue},
or $\tau = 2^{1-k}$ for small values of $k$
by the much simpler Weyl inequality~\cite{Vaughan:Book}
(see Appendix~\ref{sec:appweyl} for more information).
Improved bounds on Weyl sums are known
for intermediate values of $k$, 
but they typically take a different shape than~\eqref{eq:intro:BestWeylEst},
and therefore we do not try to incorporate them in our argument.
In conclusion, it seems that one current limitation of number-theoretic approaches
to restriction estimates for surfaces of high degree
is the poor quality of known minor arc bounds for Weyl sums.
In fact, even optimal Weyl exponents
would not allow us to obtain efficient full restriction estimates.
Fortunately, results of $\eps$-removal type ignore minor arcs to some extent,
hence the efficient ranges in those cases.

\textbf{Acknowledgements.}
The authors thank Tony Carbery, Yi Hu, Eugen Keil, Mark Lewko, Sean Prendiville and 
Trevor Wooley for stimulating discussions on restriction theory. 
The authors also thank Yuzhao Wang and Hiro Oh for pointing out~\cite{KV:Parab}. 
The work of the first author was supported by NSERC Discovery Grants
22R80520 and 22R82900.

\section{Notation}
\label{sec:notat}

For functions $f : \T^d \rightarrow \C$
and $g : \Z^d \rightarrow \C$ we define
the Fourier transforms of $f$ and $g$ by
$\wh{f}( \bfk ) = \int_{\T^d} f( \bfalpha ) e( - \bfalpha \cdot \bfk ) \dbfalpha$
and $\wh{g}( \bfalpha ) = \sum_{\bfn \in \Z^d} g(\bfn) e( \bfalpha \cdot \bfn ) $.
For a function $h : \R^d \rightarrow \C$
we define the Fourier transform by 
$\wh{h}(\bfxi) = \int_{\R^d} f(\bfx) e( - \bfxi \cdot \bfx ) \dbfx$.
For any function $f$ defined on an abelian group,
we let $\wt{f}(x) = f(-x)$.
Given a function $f : \R^d \rightarrow \R$
and two subsets $A,B$ of $\R^d$,
we write $A \prec f \prec B$
when $0 \leq f \leq 1$ everywhere,
$f = 1$ on $A$ and $f = 0$ outside $B$.
We denote the disjoint union of two sets $A$ and $B$
by $A \bigsqcup B$.

When $\mathcal{P}$ is a certain property, we
let $1_\mathcal{P}$ denote the boolean equal to $1$
when $\mathcal{P}$ holds and $0$ otherwise,
and when $E$ is a set we define the indicator function 
of $E$ by $1_E(x) = 1_{x \in E}$.
When $p \in [1,+\infty]$ is an exponent, 
we systematically denote by $p' \in [1,+\infty]$ 
its dual exponent satisfying $\frac{1}{p} + \frac{1}{p'} = 1$.
We let $\dm$ denote the Lebesgue measure on $\R^d$, or on 
$\T^d$ identified with any fundamental domain of the form $[\theta,1 + \theta)^d$,
and we let $\dSigma$ denote the counting measure on a discrete set such as $\Z^d$.
For $q \geq 2$ we occasionally use $\Z_q$ as a shorthand for the group $\Z/q\Z$.
When $N$ is an integer we write $[N] = \{ 1 , \dots , N \}$.

Throughout the article, we use the letter $\eps$ 
generically to denote a constant
which can be taken arbitrarily small, and whose value
may change in each occurence.

\section{Analytic preliminaries}
\label{sec:prelims}

In this section we discuss several standard tools in
discrete restriction theory, such as even moment bounds,
the epsilon-removal process,
and Bourgain's~\cite{Bourgain:Squares,Bourgain:ParabI} 
discrete version of the Tomas--Stein argument~\cite[Chapter~7]{Wolff:Book}
from Euclidean harmonic analysis.

We will often use a smooth weight function $\omega : \R \rightarrow [0,1]$ of the form
\begin{align}
\label{eq:prelims:Weight}
	\omega = \eta\Big( \frac{\,\cdot\,}{N} \Big),
	\qquad
	\text{$\eta$ Schwarz function such that}\
	[-1,1] \prec \eta \prec [-2,2].
\end{align}
Given a dimension $d \geq 1$, we also define the tensorized version
\begin{align}
\label{eq:prelims:WeightMultidim}
	\omega_{d}(x_1,\dots,x_d) 
	\coloneqq \omega(x_1) \cdots \omega(x_d).
\end{align}
Consider now an injective map $\bfP : \Z^d \rightarrow \Z^r$.
In a general setting, we are interested in extension
theorems for the discrete parametrized surface 
$S_N = \{ \bfP(\bfn) \,:\, \bfn \in [-N,N]^d \}$ lying in $\Z^r$.
Given a sequence $a : \Z^d \rightarrow \C$ 
supported on $[-N,N]^d$ with $\| a \|_2 = 1$
and a weight function $\omega_d : \Z^d \rightarrow [0,1]$
of the form~\eqref{eq:prelims:Weight},~\eqref{eq:prelims:WeightMultidim},
we define accordingly
\begin{align}
	\label{eq:prelims:FaDef}
	F_a(\bfalpha)
	&= \sum_{\bfn \in \Z^d} a(\bfn) e\big( \bfP(\bfn) \cdot \bfalpha \big)
	&&(\bfalpha \in \T^r),
	\\
	\label{eq:prelims:FDef}
	F(\bfalpha)
	&= \sum_{\bfn \in \Z^d} \omega_d(\bfn) e\big( \bfP(\bfn) \cdot \bfalpha \big)
	&&(\bfalpha \in \T^r),
\end{align}
which are the extension operator of our surface $S_{N}$ acting on the sequence $a$ 
and the Fourier transform of the $\omega$-smoothed counting measure on $S_{2N}$, respectively. 

For any integer $s \geq 1$,
we define $R_{s,\bfP} : \Z^r \rightarrow \C$ at $\bfu \in \Z^r$ by
\begin{align}
\label{eq:prelims:NumberRepr}
	R_{s,\bfP}(\bfu)
	= \#\{\, \bfn_1, \dots, \bfn_s \in S \,:\, \bfP(\bfn_1) + \dotsb + \bfP(\bfn_s) = \bfu \,\}.
\end{align}

We have the following well-known even moment bound:
\begin{align}
\label{eq:prelims:EvenMomentBound}
	\| F_a \|_{2s}^{2s}
	\leq \| R_{s,\bfP} \|_\infty \| a \|_2^{2s}
	\leq \| F \|_s^s \| a \|_2^{2s}.
\end{align}
This observation is occasionally useful to get
$L^2 \rightarrow L^{2s}$ from bounds on moments of unweighted exponential sums
or from arithmetic considerations on the number of representations by a system of polynomials.
To see how~\eqref{eq:prelims:EvenMomentBound} is proven, observe that
\begin{align*}
	\| F_a \|_{2s}^{2s} 
	= \| F_a^s \|_2^2
	= \int_{\T^r} \Bigg| \sum_{\bfu \in \Z^r} 
	\bigg( \sum_{ \substack{ \bfn_1,\dots,\bfn_s \in S \,: \\ 
			\bfP(\bfn_1) + \dotsb + \bfP(\bfn_s) = \bfu } } 
	a(\bfn_1) \cdots a(\bfn_s) \bigg) e( \bfalpha \cdot \bfu ) \Bigg|^2 \dbfalpha.
\end{align*}
By Plancherel and then by Cauchy-Schwarz, we deduce that
\begin{align*}
	\int_{\T^r} |F_a|^{2s} \ \dm
	\leq \sum_{\bfu \in \Z^r} R_{s,\bfP}(\bfu) 
	\sum_{ \substack{ \bfn_1,\dots,\bfn_s \in S \,: \\ 
	\bfP(\bfn_1) + \dotsb + \bfP(\bfn_s) = \bfu } } 
	|a(\bfn_1)|^2 \cdots |a(\bfn_s)|^2
	\leq \| R_{s,\bfP} \|_{\infty} \| a \|_2^{2s}.
\end{align*}
The second inequality in~\eqref{eq:prelims:EvenMomentBound} is obtained by orthogonality:
\begin{align*}
	R_{s,\bfP}(\bfu)
	\leq \sum_{ \substack{ \bfn_1,\dots,\bfn_s \in S \,: \\ 
	\bfP(\bfn_1) + \dotsb + \bfP(\bfn_s) = \bfu } } 
	\omega_d(\bfn_1) \cdots \omega_d(\bfn_d)
	= \int_{\T^r} F(\bfalpha)^s e(- \bfalpha \cdot \bfu) \dbfalpha
	\leq \|F\|_s^s.
\end{align*}

From~\cite{Bourgain:ParabI}, we recall the simple technique by which one
eliminates $\eps$-losses in restriction estimates,
using a truncated restriction estimate.

\begin{lemma}[$\eps$-removal]
\label{thm:prelims:epsremoval}
Suppose that
\begin{enumerate}
\item
$\int_{\T^r} |F_a|^p \dm 
\lesssim_\eps N^{\frac{dp}{2} - K + \eps} \|a\|_2^p$
for some $p \geq \frac{2K}{d}$,
\item
$\int_{ |F_a| \geq N^{d/2 - \zeta} \| a \|_2 } |F_a|^q \dm 
\lesssim N^{\frac{dq}{2} - K} \| a \|_2^q$
for some $q > p$ and $\zeta \in (0,\frac{d}{2})$. 
\end{enumerate}
Then 
$\int_{\T^r} |F_a|^q \dm 
\lesssim N^{\frac{dq}{2} - K} \| a \|_2^q$.
\end{lemma}

\begin{proof}
We may assume that $\| a \|_2 = 1$ by homogeneity.
We have
\begin{align*}
\int_{|F_a| \leq N^{d/2 - \zeta}} |F_a|^q \dm 	
\lesssim_\eps (N^{\frac{d}{2} - \zeta})^{q-p} \cdot N^{\frac{dp}{2} - K + \eps}
\lesssim N^{\frac{dq}{2} - K + \eps - (q-p)\zeta}.
\end{align*}
If $\eps$ is chosen small enough, we obtain a bound of the desired order
of magnitude, and the same bound for the integral over 
$\{ |F_a| \geq N^{d/2-\zeta} \}$
is already assumed to hold.
\end{proof}

We now discuss the discrete Tomas-Stein argument,
which is the starting point of 
many of our later arguments.
We introduce a parameter $\lambda > 0$ and define
\begin{align*}
E_\lambda = \{ |F_a| \geq \lambda \},
\qquad
f = 1_{E_\lambda} \frac{F_a}{|F_a|},
\qquad g = 1_{E_\lambda}.
\end{align*}
This notation will be reused in later sections.
Note that, by Cauchy-Schwarz in~\eqref{eq:prelims:FaDef},
we always have $|F_a| \leq CN^{d/2}$ (for instance one may take $C=3^d$), and thus
we assume that the parameter $\lambda$ lies in $(0,CN^{d/2}]$.

We can view the sequences $a(\bfn)$ and $\omega_d(\bfn)$
in~\eqref{eq:prelims:FaDef} and~\eqref{eq:prelims:FDef}
as functions of $\bfP(\bfn)$.
Then $F = (\omega_d 1_{S_{2N}})^\wedge$ and $F_a = (a 1_{S_N})^\wedge$,
and by Parseval, we have
\begin{align*}
	\lambda |E_\lambda|
	\leq \langle f , F_a \rangle
	= \langle f , ( a 1_{S_N} )^\wedge \rangle
	= \langle \wh{f} , a \rangle_{L^2(S_N)}.
\end{align*}
By Cauchy-Schwarz and under the assumption $\| a \|_2 = 1$, it follows that
\begin{align*}
	\lambda^2 |E_\lambda|^2
	\leq \| f \|_{L^2(S_N)}^2
	\leq \langle f \cdot \omega 1_{S_{2N}} , f \rangle.
\end{align*}
By another application of Parseval, we conclude that
\begin{align}
\label{eq:prelims:TomasStein}
	\lambda^2 |E_\lambda|^2 
	\leq \langle f \ast F, f \rangle 
	\leq \langle g \ast |F| , g \rangle.
\end{align}
We will use this inequality to obtain bounds
of the expected order on the level sets $E_\lambda$.

Via the Hardy-Littlewood method,
the kernel $F$ may typically be decomposed
into a main piece $F_\frakM$ and an error term $F_\frakm$ 
corresponding respectively to major and minor arcs,
and the Tomas-Stein argument reduces matters to obtaining operator
bounds for the convolution with $F_\frakM$ and demonstrating 
uniform power saving on $F_\frakm$. 
This strategy originated in \cite{Bourgain:Squares,Bourgain:ParabI} 
and appeared for instance in~\cite[Section~4]{Henriot:addeqs} and~\cite[Section~7]{Wooley:Restr} 
to prove $\eps$-free boundedness of the extension operator applied to the curve $(x,x^2,\dots,x^k)$;
there bounds on moments of $F_\frakM$ were used to derive operator norm bounds.
The following general lemma abstracts and generalizes this approach,
and it shows how to obtain a bound of the form~(ii) in Lemma~\ref{thm:prelims:epsremoval}
from the decomposition we just described. 

\begin{lemma}
\label{thm:prelims:TomasSteinDcp}
Suppose that there exists a decomposition $F = F_{\frakM} + F_{\frakm}$ such that
\begin{enumerate} 
\item 
$\| F_{\frakM} \ast f \|_{p} 
\lesssim 
N^{d-\frac{2K}{p}} \| f \|_{p'}$
for some $p > \frac{2K}{d}$,
\item 
$\| F_{\frakm} \|_{\infty} \lesssim N^{d(1-\tau)}$ for some $\tau \in (0,1)$. 
\end{enumerate}
Then
$\int_{|F_a| \geq N^{d/2 - \zeta} \|a\|_2} |F_a|^q \dm \lesssim N^{\frac{dq}{2} - K} \| a \|_2^q$
holds for all $q > p$ with \( \zeta = \frac{d\tau}{2} \). 
\end{lemma}

\begin{remark}
Ciprian Demeter pointed out to us that a scalar version of this lemma appears 
in the concurrent work \cite{BDG:VinoMeanValue} as an $\eps$-removal lemma for the Vinogradov main conjecture. 
One should also compare this with Lemma~6.1 of \cite{Mockenhaupt200435} and Lemma~8 of \cite{Lewko2015457} for paraboloids in the finite field setting. 
\end{remark}

\begin{remark}
We take a moment to compare this to the Keil--Zhao device in \cite{Wooley:Restr}, 
which derives its name from Theorem~4.1 of~\cite{Keil:DiophRoth}. 
The Keil--Zhao device is Tomas's original argument~\cite{Tomas:Restr}
(applied to discrete quadrics instead of the continuous sphere), 
where Keil writes out the $TT^*$ operator an equivalent expression, before applying Tomas's remarkable insight of decomposing the operator into various pieces and finding appropriate $L^1 \to L^\infty$ and $L^2 \to L^2$ bounds. 
\end{remark}

\begin{proof}
We assume again that $\| a \|_2 = 1$.
By~\eqref{eq:prelims:TomasStein},
\begin{align*}
\lambda^2 |E_\lambda|^2 
& \leq 
\| f \ast F_{\frakM} \|_p \| f \|_{p'} 
	+ \| f \ast F_{\frakm} \|_\infty \| f \|_1 
\\ 
& \lesssim 
N^{d - \frac{2K}{p}}  \| f \|_{p'}^2 
	+ \| F_{\frakm} \|_\infty \| f \|_1^2 
\\ 
&\lesssim
N^{d - \frac{2K}{p}}  | {E_\lambda} |^{\frac{2}{p'}} 
	+ N^{d - d\tau} | {E_\lambda} |^2. 
\end{align*}
Therefore, for \( \lambda \gtrsim N^{\frac{d}{2}-\frac{d\tau}{2}} \), 
\[
\lambda^2 |E_\lambda|^2 
\lesssim 
N^{d - \frac{2K}{p}} | {E_\lambda} |^{2 - \frac{2}{p}} 
. 
\]
Rearranging implies that 
$|E_\lambda| \lesssim \lambda^{-p} N^{\frac{dp}{2} - K}$.
The result then follows from the layer cake formula and our assumption $q > p$: 
\begin{align*}
\int_{|F_a| \gtrsim N^{d/2 - d\tau/2}} |F_a|^q \dm 	
& = 
q \int_{ CN^{d/2 - d\tau/2} }^{ CN^{d/2} } \lambda^{q-1} |E_\lambda| \dlambda
\\ 
&\lesssim 
N^{\frac{dp}{2} - K} \int_1^{CN^{d/2}} \lambda^{q-p-1} \dlambda 
\\
&\lesssim 
N^{\frac{dq}{2} - K}.
\end{align*}
\end{proof}

The next lemma demonstrates how incorporating subcritical estimates improves the supercritical ranges. 

\begin{lemma} 
\label{thm:prelims:SubcriticalCompletion}
Suppose that 
\begin{enumerate}
\item
$\int_{ |F_a| \geq N^{d/2 - \zeta} \|a\|_2 } |F_a|^{p_1} \dm 
\lesssim 
N^{\frac{d p_1}{2} - K} \| a \|_2^{p_1}$
for some $p_1 > \frac{2K}{d}$ and $\zeta \in (0,\frac{d}{2})$,
\item
$\int_{\T^r} |F_a|^{p_0} \dm 
\lesssim_\eps N^\eps \|a\|_2^{p_0}$
for some $p_0 \leq \frac{2K}{d}$.
\end{enumerate}
Then 
$\int_{\T^r} |F_a|^{p} \dm \lesssim N^{\frac{dp}{2} - K} \| a \|_2^p$
holds for $p > \max[\, p_1, p_0 + \zeta^{-1} ( K - \frac{d p_0}{2} ) ]$. 
\end{lemma}

\begin{proof}
We assume that $\|a\|_2 = 1$. 
The estimate of~(i) at exponent $p_1$
is also valid for exponents $p \geq p_1$
(using the trivial bound $\| F_a \|_\infty \lesssim N^{\frac{d}{2}}$), 
therefore it suffices to use the second estimate
to bound the tail
\begin{align*}
	\int_{ |F_a| \leq N^{d/2 - \zeta} } |F_a|^p \dm
	\lesssim
	(N^{\frac{d}{2} - \zeta})^{p - p_0}
	\int_{\T^r} |F_a|^{p_0} \dm
	\lesssim_\eps
	N^{\frac{dp}{2} - K} \cdot N^{K - \frac{d p_0}{2} + \eps - (p - p_0) \zeta}.
\end{align*}
This has the desired order of magnitude under our condition on $p$.
\end{proof}

This lemma has appeared implicitly in previous work,
for example with $p_0 = 2$ in~\cite[Eq.~(3.111)]{Bourgain:ParabI}, 
or with $p_0 = 4$ or $6$ in~\cite{HuLi:Degreed,HuLi:Degree3}. 
In our work, we only use Plancherel's theorem 
to exploit the subcritical estimate at $p_0 = 2$.

\section{Restriction estimates for $k$-th powers}
\label{sec:powers}

In this section, we obtain truncated restriction estimates
for the surface of $k$-th powers of integers, that is, 
for~\eqref{eq:intro:PowersSurface}.
We fix a degree $k \geq 3$,
and for a sequence $a : \Z \rightarrow \C$
supported in $[N]$ we let
\begin{align}
	\label{eq:powers:FaDef}
	F_a( \alpha )
	= \sum_{n \in \Z} a(n) e( \alpha n^k ).
\end{align}
In this section, we prove the first statement of Theorem~\ref{thm:intro:TruncRestrPowers},
as follows.

\begin{theorem}
\label{thm:powers:TruncRestrPowers}
Let $k \geq 3$ and $\tau = \max( 2^{1-k} , \frac{1}{k(k-1)} )$.
For $p > 2k$, we have
\begin{align*}
	\int_{|F_a| \geq N^{1/2 - \tau/2 + \eps} \| a \|_2} |F_a|^p \dm
	\lesssim_\eps N^{\frac{p}{2} - k} \| a \|_2^p.
\end{align*}
\end{theorem}

Before embarking on the proof, we derive
two consequences of Theorem~\ref{thm:powers:TruncRestrPowers}
mentioned in the introduction.
The first is an unconditional complete restriction estimate
corresponding to the second part of Theorem~\ref{thm:intro:TruncRestrPowers}.

\begin{corollary}
Let $k \geq 3$ and $\tau = \max( 2^{1-k} , \frac{1}{k(k-1)} )$.
For $p > 2 + 2(k - 1)/\tau$,
we have
\begin{equation}
\int_{\T} |F_a|^p \dm
\lesssim N^{\frac{p}{2} - k} \| a \|_2^p.
\end{equation}
\end{corollary}

\begin{proof}
We use Theorem~\ref{thm:powers:TruncRestrPowers} to obtain the first estimate
in the assumptions of Lemma~\ref{thm:prelims:SubcriticalCompletion},
and the trivial Plancherel estimate at $p_0 = 2$ to obtain the second one.
\end{proof}

Secondly, we obtain the whole supercritical range 
of restriction estimates under Hypothesis~$\textrm{K}$, by exploiting
conjectural estimates for even exponents of lower order.


\smallskip

\textit{Proof of Corollary~\ref{thm:intro:RestrPowersHypK}.}
Assuming Hypothesis~$\textrm{K}$ for $k$, we have 
$\| R_{k,P} \|_\infty \lesssim_\eps N^\eps$
with $P(n) = n^k$ in the notation~\eqref{eq:prelims:NumberRepr},
so that by~\eqref{eq:prelims:EvenMomentBound},
$\| F_a \|_{2k}^{2k} \lesssim_\eps N^\eps \| a \|_{2k}^{2k}$.
By Theorem~\ref{thm:powers:TruncRestrPowers} 
and Lemma~\ref{thm:prelims:epsremoval}, 
we deduce that for $p > 2k$, the following $\eps$-free estimate holds as well:
\begin{align*}
\int_{\T} |F_a|^p \ \dm \lesssim_p \| a \|_p^p.
\end{align*}
\qed

We now set out to prove Theorem~\ref{thm:powers:TruncRestrPowers}.
We fix a sequence $a : \Z \rightarrow \C$ supported on $[N]$
such that $\| a \|_2 = 1$,
and a weight function $\omega$ of the form~\eqref{eq:prelims:Weight}.
We let
\begin{align}
	\label{eq:powers:FDef}
	F( \alpha )
	= \sum_{n \in \Z} \omega(n) e( \alpha n^k ).
\end{align}
We also introduce a parameter $\lambda \in (0,N^{1/2}]$ and define
\begin{align*}
	E_\lambda = \{ |F_a| \geq \lambda \},
	\qquad
	g = 1_{E_\lambda}.
\end{align*}
We recall the Tomas-Stein inequality~\eqref{eq:prelims:TomasStein} 
from Section~\ref{sec:prelims}:
\begin{align}
\label{eq:powers:TomasStein}
\lambda^2 |E_\lambda|^2 \leq \langle g \ast |F|, g \rangle.
\end{align}

We employ the traditional Hardy-Littlewood circle method
to understand the magnitude of the exponential sum $|F|$.
We set $\tau = \max( 2^{1-k}, \frac{1}{k(k-1)} )$,
in accordance with the Weyl-type estimates of Appendix~\ref{sec:appweyl} we intend to use,
and we fix a constant $\delta = k \tau + \eps$.
For a parameter $1 \leq Q \leq N^{\delta}$,
we define the major and minor arcs of level $Q$ by
\begin{align}
	\notag
	\frakM_Q(a,q) &= \Big\lbrace \alpha \in \T \,:\, 
	\Big\| \alpha - \frac{a}{q} \Big\| \leq \frac{Q}{N^k}  \Big\rbrace, 
	\\
	\label{eq:powers:MajorArcs}
	\frakM_Q &=	\bigcup_{ q \leq Q } \bigcup_{(a,q) = 1} \frakM_Q(a,q),
	\qquad \frakm_Q = \T \smallsetminus \frakM_Q.
\end{align}

We take a few measures to simplify the exposition
in the rest of this section.
We assume implicitly that $N$
is large enough with respect to $k$ and $\delta$ 
as well as the various $\eps$ quantities for the argument to work,
without further indication.
This is certainly possible since 
Theorem~\ref{thm:powers:TruncRestrPowers} with $\| a \|_2 = 1$
is trivial for $N$ bounded (since $|F_a| \lesssim N^{d/2}$).
With these conventions in place, we now obtain two majorants for
the exponential sum $F$ on minor and major arcs
of level $Q$, via standard techniques from the circle method
recalled in Appendix~\ref{sec:appweyl}.


\begin{proposition}
\label{thm:powers:MajorMinorArcBounds}
Let $1 \leq Q \leq N^{\delta}$.
Then
\begin{align*}
	|F(\alpha)| 
	\lesssim_\eps 
	\begin{cases}
		N q^{\eps - \frac{1}{k}} ( 1 + N^k \| \alpha - \frac{a}{q} \| )^{-\frac{1}{k}}
		&\text{if $\alpha \in \frakM_Q$},
		\\
		Q^{\eps - 1/k} N
		&\text{if $\alpha \in \frakm_Q$}.
	\end{cases}
\end{align*}
\end{proposition}

\begin{proof}
Consider $a,q \in \Z$, $\beta \in \R$ 
such that $\alpha = \frac{a}{q} + \beta$, 
$1 \leq q \leq N^{k-1}$,
$(a,q) = 1$ and $|\beta| \leq \frac{1}{q N^{k-1}}$.
If $q \geq N$,
then Proposition~\ref{thm:appweyl:MinorArcBound}
with $\theta = 0$ shows that,
for $Q \leq N^\delta$,
\begin{align*}
	|F(\alpha)| 
	\lesssim_\eps N^{1 - \tau + \eps} 
	\leq Q^{-(\tau-\eps)/\delta} N
	\lesssim Q^{ \eps' - 1/k } N.
\end{align*}
Otherwise, Proposition~\ref{thm:appweyl:MajorArcBound}
shows that
\begin{align*}
	|F(\alpha)| 
	\lesssim q^{-\frac{1}{k} + \eps} N ( 1 + N^k |\beta| )^{-\frac{1}{k}}.
\end{align*}
This gives the desired bound if $\alpha \in \frakM_Q$,
and if $\alpha \not\in \frakM_Q$, then either
$q > Q$ or $|\beta| > \frac{Q}{N^k}$,
and in either case $|F(\alpha)| \lesssim_\eps Q^{\eps - 1/k} N$.
\end{proof}

We define a majorant function $V_{p,Q} : \T \rightarrow \C$ 
by\footnote{Formally, $V_{p,Q}$ also depends on $\eps$.}
\begin{align}
	\label{eq:powers:MajorantV}
	V_{p,Q} &= \sum_{q \leq Q} \ \sum_{ a \bmod q }
	q^{\eps - p/k} \tau_{-a/q} Z_p,
\end{align}
where $Z_p : \T \rightarrow \C$ is defined by
\begin{align*}
	Z_p( \theta ) =
	( 1 + N^k \| \theta \| )^{-p/k}.
\end{align*}
By Proposition~\ref{thm:powers:MajorMinorArcBounds}, we have
\begin{align}
\label{eq:powers:BoundByV}
	|F|^p \cdot 1_{\frakM_Q} \lesssim N^p \cdot V_{p,Q},
	\qquad
	\| F 1_{\frakm_Q} \|_\infty \lesssim_\eps Q^{\eps - \frac{1}{k}} N.
\end{align}
for $1 \leq Q \leq N^\delta$.
While $V_{p,Q}$ is a rather coarse majorant function,
it has the advantage that its Fourier transform 
at nonzero frequencies can be efficiently bounded:
in additive combinatorics language, it behaves pseudorandomly. 
This can be used in turn to obtain efficient $L^2 \rightarrow L^2$
bounds for the operator of convolution with $V_{p,Q}$.
This was the approach taken by Bourgain~\cite{Bourgain:Squares}
in the case of squares $k = 2$.
We follow this approach and start by bounding the Fourier transform of 
the majorant $V_{p,Q}$ with the help of the truncated divisor functions
$d(\ell,Q) = \sum_{ n \leq Q \,:\, n | \ell } 1$.

\begin{proposition}
\label{thm:powers:MajorantVBounds}
If $p > k$, we have
\begin{align}
	\label{eq:powers:MajorantVDivBound1}
	&\phantom{\forall\ \ell \neq \mathbf{0}} &
	|\wh{V}_{p,Q}(\ell)| 
	&\lesssim_p
	N^{-k} d( \ell, Q )
	&&(\ell \in \Z).
\end{align}
If $p  = k$, we have
\begin{align}
	\label{eq:powers:MajorantVDivBound2}
	&\phantom{(\ell \in \Z \smallsetminus \{0\})} &
	|\wh{V}_{p,Q}(\ell)| &\lesssim_\eps N^{\eps - k}
	&&(\ell \in \Z \smallsetminus \{0\}),	
	\\
	\label{eq:powers:MajorantVDivBound3}
	&&
	\wh{V}_{p,Q}(0)
	&\lesssim_\eps Q N^{\eps - k}.
	&&
\end{align}
\end{proposition}

\begin{proof}
By a linear change of variables, we have
\begin{align*}
	\int_{\T} Z_p(\theta) \dtheta
	\lesssim
	\int_0^1
	( 1 + N^k \theta )^{-p/k} 
	\dtheta	
	\lesssim N^{-k} \int_{\R} ( 1 + |\xi| )^{-p/k} \dxi.
\end{align*}
By a spherical change of coordinates, we see therefore that
\begin{align}
\label{eq:powers:L1boundZ}
	\| Z_p \|_1
	\lesssim
	\begin{cases}
		C_p N^{-k}				& \text{if $p > k$} \\
		C_\eps N^{\eps - k} 	& \text{if $p = k$}.
	\end{cases}
\end{align}
Recalling~\eqref{eq:powers:MajorantV}, 
this can be used to estimate $V_{p,Q}$ in $L^1$ when $p = k$:
\begin{align*}
	\| V_{p,Q} \|_1 
	\leq
	\sum_{q \leq Q} q^{\eps + 1 - p/k} \| Z_p \|_1
	\lesssim_\eps Q N^{\eps - k}.
\end{align*}

Performing Fourier inversion in~\eqref{eq:powers:MajorantV},
we obtain also
\begin{align*}
	\wh{V}_{p,Q}( \ell ) 
	&= \sum_{ q \leq Q } q^{\eps - p/k}
	\sum_{ a \in \Z_q } e_q( -a \ell ) \wh{Z}_p(\ell).
\end{align*}
By orthogonality it follows that
\begin{align*}
	\wh{V}_{p,Q}( \ell ) 
	= \bigg( \sum_{ \substack{ q \leq Q \\ q | \ell }} q^{\eps + 1 - p/k} \bigg) \wh{Z}_p( \ell ).
\end{align*}
The sum inside the parenthesis is bounded by $N^{\eps} d( \ell, Q )$ if $p = k$ and 
by $d( \ell, Q )$ if $p > k$ and $\eps$ is small enough with respect to $p$.
Using also $\|\wh{Z}_p\|_\infty \leq \| Z_p \|_1$ 
and the estimate~\eqref{eq:powers:L1boundZ}, this concludes the proof.
\end{proof}

We begin by removing the minor arcs contribution to 
the expression~\eqref{eq:powers:TomasStein}, and
we use $L^p$ norms to estimate the remaining piece.

\begin{proposition}
\label{thm:powers:ChoppingMinorArcs}
Suppose that $\eta^{- 2k - \eps} \leq Q \leq N^{\delta}$.
Then, for $p > 0$,
\begin{align}
\label{eq:powers:ChoppingMinorArcs}
	\eta^{2p} |E_{\eta N^{1/2}}|^2 \lesssim_\eps \langle V_{p,Q} , g \ast \wt{g} \rangle.
\end{align}
\end{proposition}

\begin{proof}
By~\eqref{eq:powers:TomasStein},
and Hölder's inequality, it follows that
\begin{align*}
	\lambda^2 |E_{\lambda}|^2 
	&\leq
	\int_\T |F| 1_{\frakM_Q} d(g \ast \wt{g}) +
	\langle ( |F| 1_{\frakm_Q} ) \ast g , g \rangle
	\\
	&\leq
	\| F 1_{\frakM_Q} \|_{L^p(d(g \ast \wt{g}))} 
	\cdot \| 1 \|_{L^{p'}(d(g \ast \wt{g}))}
	+ \| ( |F| 1_{\frakm_Q} )  \ast g \|_\infty \| g \|_1 
	\\
	&\leq
	\langle |F|^p 1_{\frakM_Q} , g \ast \wt{g} \rangle^{\frac{1}{p}} \cdot |E_\lambda|^{2 - \frac{2}{p}}  +
	\| F 1_{\frakm_Q} \|_\infty |E_\lambda|^2.
\end{align*}
Inserting the estimates of~\eqref{eq:powers:BoundByV} 
and assuming that $\lambda^2 \geq Q^{\eps - 1/k} N$, we obtain
\begin{align*}
	\lambda^{2p} |E_{\lambda}|^2
	\lesssim N^p \langle V_{p,Q} , g \ast \wt{g} \rangle.
\end{align*}
The proof if finished upon writing $\lambda = \eta N^{1/2}$.
\end{proof}

We can now derive our first level set estimate,
which features an $N^\eps$ term.

\begin{proposition}
\label{thm:powers:LevelSetEpsFull}
Let $\zeta = \frac{\delta}{2k}$.
For $p = k$, we have
\begin{align}
\label{eq:powers:LevelSetEpsFull}
	|E_{\eta N^{1/2}}| 
	\lesssim_\eps N^{\eps - k} \eta^{-2p}
	\quad\text{for}\quad
	\eta \geq N^{- \zeta + \eps}.
\end{align}
\end{proposition}

\begin{proof}
We assume that $\eta \geq N^{-\delta/2k + \eps}$
and let $Q = \eta^{- 2k - \eps}$.
By Proposition~\ref{thm:powers:ChoppingMinorArcs} with $p = k$
and Fourier inversion, it follows that
\begin{align*}
	\eta^{2k} |E_\lambda|^2
	&\leq \langle \wh{V}_{k,Q} , |\wh{g}|^2 \rangle 
	\\
	&\leq |\wh{V}_{k,Q}(0)| \, |\wh{g}(0)|^2
	+ \| \wh{V}_{k,Q} 1_{\Z^t \smallsetminus \{0\}} \|_\infty \| \wh{g} \|_2^2.
\end{align*}
By Plancherel,~\eqref{eq:powers:MajorantVDivBound2} 
and~\eqref{eq:powers:MajorantVDivBound3},
we obtain
\begin{align*}
	\eta^{2k} |E_\lambda|^2 \lesssim_\eps
	Q N^{\eps - k} |E_\lambda|^2
	+ N^{\eps - k} |E_\lambda|.
\end{align*}
We have $\eta \geq N^{-\delta/2k} \geq N^{-1/4 + \eps}$,
and therefore $\eta^{2k} \geq Q N^{\eps - k}$,
so that
\begin{align*}
	\eta^{2k} |E_\lambda|^2
	\lesssim_\eps
	N^{\eps - k} |E_\lambda| 
	\quad\Rightarrow\quad
	|E_\lambda| \lesssim_\eps N^{\eps - k} \eta^{-2k}.
\end{align*}
\end{proof}

We now obtain a level set estimate designed to
remove the $N^\eps$ that arises in using 
Proposition~\ref{thm:powers:LevelSetEpsFree}
to bound the moments of $F_a$.
We first introduce a technical tool to
keep track of the information 
that the Fourier transform of $F$ has support in $[N^k]$.
Consider a non-negative trigonometric polynomial $\psi_N$
such that $[-N^k,N^k] \prec \wh{\psi}_N \prec [-2N^k,2N^k]$,
then $\int_\T \psi_N = \wh{\psi}_N(0) = 1$.
By Fourier inversion, we can see that $F = F \ast \psi_N$.
Starting from~\eqref{eq:powers:TomasStein}, it is then easy to obtain
the following analogue of Proposition~\ref{thm:powers:ChoppingMinorArcs}.

\begin{proposition}
\label{thm:powers:ChoppingMinorArcs2}
Suppose that $\eta^{- 2k - \eps} \leq Q \leq N^{\delta}$.
Then, for $p > 0$,
\begin{align*}
	\eta^{2p} |E_{\eta N^{1/2}}|^2 \lesssim_{\eps} \langle V_{p,Q} , g \ast \wt{g} \ast \wt{\psi}_N \rangle.
\end{align*}
\end{proposition}

At this stage, we need to import a divisor bound 
used by Bourgain~\cite{Bourgain:Squares}.

\begin{proposition}
\label{thm:powers:DivBound}
Let $B \geq 1$ be an integer,
and suppose that $1 \leq Q \leq N^{k/B}$.
Then
\begin{align}
\label{eq:powers:DivBound}
	\sum_{ |\ell| \leq 2N^k }	
	d( \ell, Q )^B
	\lesssim_{\eps,B} Q^\eps N^k.
\end{align}
\end{proposition}

\begin{proof}
In the sum of~\eqref{eq:powers:DivBound},
the term $\ell=0$ contributes at most $Q^B$,
and by~\cite[eq.~(4.31)]{Bourgain:Squares} the other terms 
contribute at most $C_{\eps,B} Q^\eps N^k$.
The conclusion follows from our assumption on $Q$.
\end{proof}

We now proceed to our $\eps$-removal level set estimate.

\begin{proposition}
\label{thm:powers:LevelSetEpsFree}
Let $\nu \in (0,1]$ be a parameter.
There exists a constant $c_\nu > 0$ such that,
for $p > k$,
\begin{align*}
	|E_{\eta N^{1/2}}| \lesssim_\nu N^{-k} \eta^{-2(1+\nu)p}
	\qquad\text{for} 
	\qquad \eta \geq N^{-c_{\nu}}.
\end{align*}
\end{proposition}

\begin{proof}
We assume again that $C \eta^{-2k - \eps} \leq Q \leq N^{\delta}$,
and we apply Proposition~\ref{thm:powers:ChoppingMinorArcs2}
for a fixed $p > k$.
By Proposition~\ref{thm:powers:ChoppingMinorArcs2}, we have
\begin{align*}
	\eta^{2p} |E_{\eta N^{1/2}}|^2 
	\lesssim \langle V_{p,Q} \ast \psi_N, g \ast \wt{g} \rangle
	= \langle \wh{V}_{p,Q} \, \wh{\psi}_N, |\wh{g}|^2 \rangle.
\end{align*}
Applying~\eqref{eq:powers:MajorantVDivBound1} and the bound
$\| \wh{\psi}_N \|_\infty \leq \int \psi_N = 1$, 
we deduce that
\begin{align*}
	\eta^{2p} |E_{\eta N^{1/2}}|^2 
	&\lesssim N^{-k} 
	\sum_{ |\ell| \leq 2N^k }
	d( \ell, Q ) |\wh{g}(\ell)|^2.
\end{align*}
Let $q,q' \in [1,\infty]$ be a dual pair of exponents
to be determined later.
Assuming that $q \in \N$ and $Q \leq N^{k/q}$,
applications of Hölder and Proposition~\ref{thm:powers:DivBound} furnish
\begin{align*}
	\eta^{2p} |E_{\eta N^{1/2}}|^2 
	&\lesssim N^{-k} 
	\bigg[ \sum_{ |\ell| \leq 2N^k } d( \ell, Q )^q \bigg]^{\frac{1}{q}}	
	\bigg[ \sum_{ |\ell| \leq 2N^k } |\wh{g}(\ell)|^{2q'} \bigg]^{\frac{1}{q'}}	
	\\
	&\lesssim_{q,\eps} N^{-k}
	( Q^\eps N^k )^{1/q} \| \wh{g} \|_\infty^{2 - 2/q'} ( \| \wh{g} \|_2^2 )^{1/q'}
	\\
	&\leq N^{ - k/q'} Q^{\eps/q} |E_{\eta N^{1/2}}|^{2 - 1/q'}.
\end{align*}
Rearranging terms in the above, we find that
\begin{align*}
	|E_{\eta N^{1/2}}| \lesssim_{q,\eps} N^{-k} Q^{\eps q'/q} \eta^{-2 q' p}.
\end{align*}
Choose finally $Q = C \eta^{- 2k - 2\eps}$
and $q$ a large enough integer so that $q' < 1 + \nu$,
and note that we obtain the desired bound 
for $\eps$ small enough.
The condition $Q \leq N^{k/q}$ is satisfied
for $\eta \geq N^{-c_\nu}$ with a certain $c_\nu > 0$.
\end{proof}

We now prove Theorem~\ref{thm:powers:TruncRestrPowers} 
(with the rescaling $\| a \|_2 = 1$) by integrating
the previous level set estimates.

\begin{proposition}
\label{thm:powers:RestrEpsFullFree}
Let $\zeta = \frac{\delta}{2k}$.
We have
\begin{align}
\label{eq:powers:RestrEpsFull}
	\int_{|F_a| \geq N^{1/2 - \zeta + \eps}} |F_a|^p \ \dm
	\lesssim_\eps N^{\frac{p}{2} - k + \eps}
	\quad\text{for}\quad
	p \geq 2k.	
\end{align}
There exists $c_\nu > 0$ such that, for $p > 2k$,
\begin{align}
\label{eq:powers:RestrEpsFree}
	\int_{|F_a| \geq N^{-c_\nu}} |F_a|^p \ \dm
	\lesssim_p N^{\frac{p}{2} - k}
	\quad\text{for}\quad
	p > 2k.
\end{align}
\end{proposition}

\begin{proof}
By the layer cake formula and
Proposition~\ref{thm:powers:LevelSetEpsFull}, we obtain
\begin{align*}
	\int_{ |F_a| \geq N^{1/2 - \zeta + \eps} } |F_a|^p \ \dm
	&\asymp N^{p/2} \int_{N^{- \zeta + \eps}}^1 \eta^{p-1} |E_{\eta N^{1/2}}| \deta
	\\
	&\lesssim_\eps N^{\frac{p}{2} - k + \eps} \int_{N^{- \zeta + \eps}}^1 \eta^{p - 2k - 1} \deta
	\\
	&\lesssim_\eps
	N^{\frac{p}{2} - k + \eps}, \phantom{\int_0^1}
\end{align*}
where $p \geq 2k$ ensured that the $\eta$-integral is $\lesssim \log N$.

The second estimate is obtained similarly,
by invoking Proposition~\ref{thm:powers:LevelSetEpsFree}
in place of Proposition~\ref{thm:powers:LevelSetEpsFull}.
\end{proof}

\textit{Proof of Theorem~\ref{thm:powers:TruncRestrPowers}.}
Remember that $\delta > k\tau$ was arbitrary, and therefore 
the parameter $\zeta$ in Proposition~\ref{thm:powers:RestrEpsFullFree}
can be given a value arbitrarily close to $\frac{\tau}{2}$.
By a minor variant of Lemma~\ref{thm:prelims:epsremoval},
we can now use~\eqref{eq:powers:RestrEpsFree} to
remove the $N^\eps$ factor in~\eqref{eq:powers:RestrEpsFull} for $p > 2k$.
\qed

\section{Extending the moment method}
\label{sec:exten}

The method of the previous section extends to
many surfaces, due to its reliance on little
number-theoretic information.
However, it does not seem to produce
truncated restriction estimates in the complete
supercritical range for many interesting cases,
and therefore we only sketch this class of results.

Fix $t \geq 1$ and a tuple of integers $\bfk \in \Z^t$
with $1 \leq k_1 < \dotsb < k_t$.
We consider the monomial curve
\begin{align*}
	S = \{ (n^{k_1},\dots,n^{k_t}) \,:\, n \in [N] \}.
\end{align*}
Define also the maximal degree $k = k_t$
and the total degree $K = k_1 + \dotsb + k_t$.
For a sequence $a : \Z \rightarrow \C$ supported on $[N]$
define the following exponential sums associated to $S$:
\begin{align*}
	&\phantom{(\bfalpha \in \T^t)} &
	F( \bfalpha ) 
	&= \sum_{n \in [N]} e( \alpha_1 n^{k_{t_1}} + \dotsb + \alpha_t n^{k_t} )
	&&(\bfalpha \in \T^t),
	\\
	&\phantom{(\bfalpha \in \T^t)} &
	F_a( \bfalpha ) 
	&= \sum_{n \in \Z} a(n) e( \alpha_1 n^{k_{t_1}} + \dotsb + \alpha_t n^{k_t} )
	&&(\bfalpha \in \T^t).
\end{align*}
It can be checked that the method of Section~\ref{sec:powers} 
yields truncated restriction exponents in the range $p > 2kt$, which is quite far from
the full supercritical range $p > 2(k_1 + \dotsb + k_t)$ for large values of $t$.
It turns out to be more useful to use a different majorant in that situation.
We only describe the main steps of this variant since it was already derived
in the case $\bfk = (1,\dots,k)$ in previous work
(see~\cite[Section~4]{Henriot:addeqs} and~\cite[Section~7]{Wooley:Restr}).
By the circle method, one can obtain
a decomposition of the form $F = F_\frakM + F_\frakm$ with 
\begin{align}
\label{sec:exten:alala}
	\| F_\frakM \|_p^p \lesssim \frakS_p \cdot \frakJ_p \cdot N^{p-K},
	\qquad
	\| F_\frakm \|_\infty
	\lesssim_\eps N^{1-\tau+\eps},
\end{align}
where $\frakS_p$ and $\frakJ_p$
are respectively the singular series
and the singular integral defined by
\begin{align*}
	\frakS_p 
	&= \sum_{q \geq 1 } \sum_{(\bfa,q) = 1} 
	\bigg| \sum_{(\bfa,q) = 1} e_q( a_1 u^{k_1} + \dotsb + a_t u^{k_t} ) \bigg|^p,
	\\
	\frakJ_p
	&= \int_{\R^t} \bigg| \int_{\R} e( \xi_1 x^{k_1} + \dotsb + \xi_t x^{k_t} ) \dx \bigg| \dbfxi.
\end{align*}
It is known from classical work of Hua~\cite{Hua:SgSeriesIntg}
and Arkhipov-Chubarikov-Karatsuba~\cite[Theorems~1.3,~1.4,~2.4,~2.5]{ACK:Book},
that when $\bfk = (1,\dots,k)$, 
$\frakJ_p < \infty$  for $p > K + 1$ and $\frakS_p < \infty$ for $p > K+2$, 
while when  $\bfk \neq (1,\dots,k)$ and $k \geq 4$,
$\frakJ_p < \infty$ for $p > K$
and $\frakS_p < \infty$ for $p > K + 1$.
Via Lemma~\ref{thm:prelims:TomasSteinDcp}, and writing $\rho = \tau/2$, this gives
\begin{align*}
	\int_{|F_a| \geq N^{1/2 - \rho}} |F_a|^q \dm
	\lesssim N^{\frac{q}{2} - K}
	\quad\text{for $q > 2K + 2$}
\end{align*}
if $\bfk \neq (1,\dots,k)$ and $k \geq 4$, and
\begin{align*}
	\int_{|F_a| \geq N^{1/2 - \rho}} |F_a|^q \dm
	\lesssim N^{\frac{q}{2} - K}
	\quad\text{for $q > 2K + 4$}
\end{align*}
if $\bfk = (1,\dots,k)$.
(This last estimate is the one that was already obtained 
in~\cite{Henriot:addeqs} and~\cite{Wooley:Restr}).
Note that the above ranges of exponent miss the conjectured
ones by two or four variables only.

\section{Arc mollifiers}
\label{sec:mollif}

This section serves to introduce a technical tool,
borrowed from Bourgain~\cite[Section~3]{Bourgain:ParabI}
and used in the proof of 
Theorems~\ref{thm:intro:TruncRestrParabHighDim} and~\ref{thm:intro:TruncRestrParabLowDim}.
It consists in a collection of multipliers in
the frequency variable $\alpha \in \T$,
which serves as a partition of unity adapted to the major arcs,
that is, the collection of small neighborhoods of rationals with small denominator.
We recall the natural bounds on these multipliers and their Fourier transform.
Throughout the section we fix an integer $k \geq 3$,
which corresponds to the degree $k$ of the $k$-paraboloid 
in Sections~\ref{sec:parabhigh} and~\ref{sec:parablow}.

We fix a smooth bump function $\kappa$ 
with $[-1,1] \prec \kappa \prec [-2,2]$.
Let $\wt{N} = 2^{\lfloor \log_2 N \rfloor}$, 
and for every integer $0 \leq s \leq \lfloor \log_2 N \rfloor$ define
\begin{align}
\label{eq:mollif:phiDef}
	\phi^{(s)} \coloneqq
	\begin{cases}
	\kappa( 2^s N^{k-1} \,\cdot\, ) - \kappa( 2^{s+1} N^{k-1} \,\cdot\, )
	&	\text{if $1 \leq 2^s < \wt{N}$},
	\\
	\kappa( 2^{s} N^{k-1} \, \cdot \,)
	&	\text{if $2^s = \wt{N}$}.
	\end{cases}
\end{align}
Note that we have
\begin{align}
\label{eq:mollif:phiSupport}
	\Supp( \phi^{(s)} ) \subset
	\begin{dcases}	
	\pm \bigg[ \frac{1}{2^{s+1} N^{k-1}}, \frac{1}{2^{s-1}N^{k-1}} \bigg]
	&	\text{if $1 \leq 2^s < \wt{N}$}.
	\\
	\bigg[\! - \frac{1}{2^{s-1} N^{k-1}}, \frac{1}{2^{s-1}N^{k-1}} \bigg]
	&	\text{if $1 \leq 2^s \leq \wt{N}$},
	\end{dcases}
\end{align}
More importantly, for every dyadic integer $1 \leq Q \leq N$, we have
\begin{align}
\label{eq:mollif:PartUnity}
	\sum_{Q \leq 2^s \leq N} \phi^{(s)}
	= \kappa( Q N^{k-1} \,\cdot\,).
\end{align}

We let $N_1 = c_1 N$, for a small constant $c_1 \in (0,1]$.
It is then easy to check that the intervals
\begin{align}
\label{eq:mollif:Intervals}
	\frac{a}{q} + \bigg[\! -\frac{2}{QN^{k-1}} , \frac{2}{QN^{k-1}} \bigg],
	\quad 1 \leq a \leq q,\ q \sim Q,\ 1\leq Q \leq N_1
\end{align}
are all disjoint.
For a dyadic integer $Q$ and an integer $0 \leq s \leq \log_2 N$,
we define the arc mollifier
\begin{align}
\label{eq:mollif:PhiDef}
	\Phi_{Q,s} = \sum_{\substack{ (a,q) = 1 \\ q \sim Q }}
	\tau_{-a/q} \phi^{(s)},
\end{align}
so that, by~\eqref{eq:mollif:phiSupport} and disjointness,
\begin{align}
\label{eq:mollif:PhiSupport}
	\Supp( \Phi_{Q,s} ) \subset
	\begin{dcases}
		\bigsqcup_{\substack{ (a,q) = 1 \\ q \sim Q }} 
		\bigg( \frac{a}{q} \pm \bigg[ \frac{1}{2^{s+1} N^{k-1}}, \frac{1}{2^{s-1}N^{k-1}} \bigg] \bigg)		
		&	\text{for $Q \leq 2^s < \wt{N}$},
		\\
		\bigsqcup_{\substack{ (a,q) = 1 \\ q \sim Q }} 
		\bigg( \frac{a}{q} + \bigg[ \! - \frac{1}{2^{s-1} N^{k-1}}, \frac{1}{2^{s-1}N^{k-1}} \bigg] \bigg)
		&	\text{for $Q \leq 2^s \leq \wt{N}$}.
	\end{dcases}
\end{align}
We finally define
\begin{align}
\label{eq:mollif:rhoDef}
	\lambda = \sum_{Q \leq N_1} \sum_{Q \leq 2^{s} \leq N} \Phi_{Q,s},
	\qquad
	\rho = 1 - \lambda.
\end{align}

\begin{proposition}
\label{thm:mollif:rhoSupport}
We have $0 \leq \lambda, \rho \leq 1$ and 
\begin{align*}
	\lambda = 1,\ \rho = 0
	\quad\text{on}\quad
	\bigsqcup_{Q \leq N_1} \bigsqcup_{\substack{ (a,q) = 1 \\ q \sim Q}} 
	\bigg( \frac{a}{q} + \bigg[ \! -\frac{1}{QN^{k-1}}, \frac{1}{QN^{k-1}} \bigg] \bigg).
\end{align*}
\end{proposition}

\begin{proof}
By~\eqref{eq:mollif:PartUnity}, we can rewrite $\lambda$ as
\begin{align*}
	\lambda = \sum_{Q \leq N_1} \sum_{\substack{(a,q) = 1 \\ q \sim Q}}
	\tau_{-a/q} \bigg( \sum_{Q \leq 2^s \leq N} \phi^{(s)} \bigg)
	= \sum_{Q \leq N_1} \sum_{\substack{(a,q) = 1 \\ q \sim Q}} \tau_{-a/q} \kappa(QN^{k-1} \,\cdot\,).
\end{align*}
The proposition follows since we assumed 
that $[-1,1] \prec \eta \prec [-2,2]$
and the intervals~\eqref{eq:mollif:Intervals} are disjoint.
\end{proof}

At this stage we define the fundamental domain 
$\calU = ( \frac{1}{2N_1} , 1 + \frac{1}{2N_1}]$,
and we note that when $N$ is large, 
then for every $1 \leq a \leq q \leq Q \leq N_1$,
we have 
\begin{align*}
	\frac{a}{q} + \bigg[ \! - \frac{2}{QN^{k-1}} , \frac{2}{QN^{k-1}} \bigg] \subset \overset{\circ}{\calU}
\end{align*}
Therefore for $1 \leq Q \leq 2^s \leq N$, the functions 
$\phi^{(s)}$, $\Phi_{Q,s}$ and $\lambda$ are supported on the interior of $\calU$,
and they may be viewed as smooth functions over the torus $\T$,
by $1$-periodization from the interval $\calU$.
We will view $\Phi_{Q,s}$ alternatively as a smooth function
on the torus $\T$ or on the real line, but note that for an integer $n$,
$\wh{\Phi_{Q,s}}(n)$ has the same definition under both points of view.

For $n \in \Z$ and an integer $Q \geq 1$ we define
\begin{align*}
	d(n,Q) = \sum_{\substack{ \ 1 \leq d \leq Q \,: \\ d|n }} 1.
\end{align*}
The following useful lemma is due to Bourgain~\cite{Bourgain:ParabI}.
We include the short proof for completeness.

\begin{lemma}
\label{thm:mollif:DivBound}
Let $\delta_x$ be the Dirac function at $x$. Then
\begin{align*}
	&\phantom{(n \in \Z)} &
	\wh{ \sum_{\substack{ (a,q) = 1 \\ q \sim Q}} \delta_{a/q} }(n)
	\lesssim Q \cdot d(n,2Q)
	&&(n \in \Z).
\end{align*}
\end{lemma}

\begin{proof}
We note that 
$\sum_{(a,q) = 1} \wh{\delta_{a/q}}(n)
= \sum_{(a,q) = 1} e_q( a n ) = c_q(n)$
is a Ramanujan sum.
By a well-known convolution identity~\cite[Theorem~4.1]{MV:Book}, 
we have then
\begin{align*}
	\bigg| \sum_{q \sim Q} c_q(n) \bigg|
	= \bigg| \sum_{q \sim Q} \sum_{d | (q,n)} d \mu(q/d) \bigg|
	\leq \sum_{\substack{ d | n \\ d \leq 2Q}} \ d \ \sum_{\substack{ q \sim Q \\ d | q }} 1
	\lesssim Q \sum_{\substack{ d | n \\ d \leq 2Q}} 1,
\end{align*}
and the leftmost term above is exactly $| \sum_{q \sim Q} \sum_{(a,q)=1} \wh{\delta_{a/q}}(n) |$.
\end{proof}

\begin{proposition}
We have
\begin{align}
	\label{eq:mollif:PhiAverage}
	&\phantom{(m \in \Z)} &
	\int \Phi_{Q,s} \dm
	&\lesssim \frac{Q^2}{2^s N^{k-1}},
	&&
	\\
	\label{eq:mollif:PhiFourierBound}
	&\phantom{(n \in \Z)} &
	\wh{\Phi_{Q,s}}(n) 
	&\lesssim \frac{Q}{2^s N^{k-1}} d(n,2Q)
	&&(n \in \Z)
\end{align}
\end{proposition}

\begin{proof}
Let $\gamma^{(s)} = \kappa - \kappa(2 \,\cdot\,)$ for $0 \leq s < \lfloor \log_2 N \rfloor$
and $\gamma^{(s)} = \kappa$ when $s = \lfloor \log_2 N \rfloor$.
By~\eqref{eq:mollif:phiDef} and~\eqref{eq:mollif:PhiDef}, 
we can write 
\begin{align*}
	\Phi_{Q,s}	
	= \sum_{\substack{(a,q)=1 \\ q \sim Q}} \tau_{-a/q} \gamma^{(s)}(2^s N^{k-1} \,\cdot\, )
	= \bigg( \sum_{\substack{(a,q)=1 \\ q \sim Q}} \delta_{a/q} \bigg) 
	\ast  \gamma^{(s)}(2^s N^{k-1} \,\cdot\, ).
\end{align*}
From Lemma~\ref{thm:mollif:DivBound}, we deduce
the pointwise bound
\begin{align*}
	|\wh{\Phi_{Q,s}}(n)|
	=  \bigg| \wh{\sum_{\substack{(a,q)=1 \\ q \sim Q}} \delta_{a/q}}(n)
	\cdot \frac{1}{ 2^s N^{k-1} } \wh{\gamma^{(s)}}\Big( \frac{n}{2^s N^{k-1}} \Big)  \bigg|
	\lesssim \frac{Q}{ 2^s N^{k-1} } d(n,2Q),
\end{align*}
which is uniform in $n \in \Z$.
When $n = 0$ the left-hand side is $\int \Phi_{Q,s} \dm$.
\end{proof}

\begin{proposition}
For every $\eps > 0$ and $A > 0$, we have
\begin{align}
	\label{eq:mollif:rhoAverage}
	\int \rho \,\dm &\asymp 1,
	\\
	\label{eq:mollif:rhoFourierBound}
	\wh{\rho}(n) &\lesssim_{\eps,A} \frac{1}{N^{k-1-\eps}}
	\quad\text{for $0 < |n| \leq A N^A$.}
\end{align}
\end{proposition}

\begin{proof}
From~\eqref{eq:mollif:rhoDef}
and~\eqref{eq:mollif:PhiAverage},
it follows that
\begin{align*}
	\int \rho \,\dm 
	&= 1 - O\bigg(  \sum_{Q \leq N_1} \sum_{Q \leq 2^s \leq N} \frac{Q^2}{2^s N^{k-1}} \bigg)
	\\
	&= 1 - O\bigg( \frac{1}{N^{k-1}} \sum_{Q \leq N_1} Q \bigg)
	\\
	&= 1 - O\Big( \frac{N_1}{N^{k-1}} \Big).
\end{align*}
Since we have chosen $N_1 = c_0 N$ with $c_1$ small enough,
we have $\int \rho \dm \asymp 1$ as desired.
The bound on $\wh{\rho}$ is derived from~\eqref{eq:mollif:PhiFourierBound} 
in a similar fashion, using also the standard
divisor bound $d(n,Q) \leq d(n) \lesssim_\eps n^\eps$.
\end{proof}

\section{Restriction estimates for $k$-paraboloids of arbitrary dimension} 
\label{sec:parabhigh}

In this section, we obtain truncated restriction estimates
for the surface~\eqref{eq:intro:ParabSurface}, 
for an arbitrary dimension $d \geq 1$
and degree $k \geq 3$.
For simplicity, we write $|\bfx|_k = (x_1^k + \dots + x_d^k)^{1/k}$
for vectors $\bfx \in \R^d$;
this quantity may be negative when $k$ is odd.
Note that the system of polynomials $\bfP = (\bfx,|\bfx|_k^k)$
has total degree $K = d+k$, and therefore the critical restriction exponent
is $p_{d,k} = \frac{2(d+k)}{d}$ for the surface~\eqref{eq:intro:ParabSurface}.
For a sequence $a : \Z^d \rightarrow \C$ supported on $[-N,N]^d$, we let
\begin{align}
\label{eq:parabhigh:FaDef}
	F_a(\alpha,\bftheta)
	&= \sum_{\bfn \in \Z^d} a(\bfn) e( \alpha |\bfn|_k^k + \bftheta \cdot \bfn )
	&&(\alpha \in \T, \bftheta \in \T^d).
\end{align}

The following estimate, a slightly more precise
version of the first statement in Theorem~\ref{thm:intro:TruncRestrParabHighDim}, 
is the main result of this section.
Note that we miss the complete supercritical range
by a term of size $\frac{2k}{d}$, but we obtain a uniform result
for all dimensions~$d$ and degrees~$k$.

\begin{theorem}
\label{thm:parabhigh:TruncRestrParab}
Suppose that $d \geq 1$ and $k \geq 3$,
and let $\tau = \max(2^{1-k},\frac{1}{k(k-1)})$.
For every $p > \frac{2(d+k)}{d} + \frac{2k}{d}$ and $\eps > 0$, we have
\begin{align*}
	\int_{|F_a| \geq N^{d/2 - d\tau/2 + \eps} \|a\|_2 } |F_a|^p \dm
	\lesssim_{p,\eps} N^{\frac{dp}{2} - (d + k)} \| a \|_2^p.
\end{align*}
\end{theorem}

We record below the corresponding restriction estimate
that can be obtained by bounding the tail of the integral.

\begin{corollary}
\label{thm:parabhigh:RestrParab}
Suppose that $d \geq 1$ and $k \geq 3$,
and let $\tau = \max(2^{1-k},\frac{1}{k(k-1)})$.
The restriction estimate $\int |F_a|^p \dm \lesssim N^{\frac{dp}{2} - (d + k)} \|a\|_2^p$
holds for $p > 2 + \frac{2k}{d\tau}$.
\end{corollary}

\begin{proof}
We invoke Lemma~\ref{thm:prelims:SubcriticalCompletion}. 
The first assumption is verified with $\zeta \leftarrow \frac{d\tau}{2} - \eps$ for any
$p_1 > \frac{2(d + 2k)}{d}$  by Theorem~\ref{thm:parabhigh:TruncRestrParab}
and the second is verified for $p_0 = 2$ by Plancherel.
Since $2 + \frac{2k}{d\tau} \geq \frac{2(d+2k)}{d}$, we obtain
a range of exponents $p > 2 + \frac{2k}{d\tau}$.
\end{proof}

Our argument will make use of Lemma~\ref{thm:prelims:TomasSteinDcp}, 
whose philosophy borrows from the circle method the paradigm of 
major arc and minor arc estimates. 
As such we will split our convolution kernel $F$, defined in~\eqref{eq:parabhigh:FDef} below,
into major arc pieces and a minor arc piece. 
On the minor arc piece we will only need some power savings on the trivial bound. 
We decompose the major arc pieces in a fashion similar to~\cite{Bourgain:ParabI},
but simpler, and we use the Tomas--Stein method to obtain decent estimates.

We introduce some notation before turning to our proof. 
We fix integers $d \geq 1$ and $k \geq 3$ throughout,
on which every implicit or explicit constant throughout is allowed depend.
The letter $Q$ will always denote an integer
of the form $2^r$ with $r \geq 0$.
We also fix weight functions $\omega$ and $\omega_d$ of the 
form~\eqref{eq:prelims:Weight} and~\eqref{eq:prelims:WeightMultidim},
and we define the exponential sums
\begin{align}
	\label{eq:parabhigh:FDef}
	F( \alpha, \bftheta ) &= \sum_{ \bfn \in \Z^d } \omega_d(\bfn) e( \alpha |\bfn|_k^k + \bftheta \cdot \bfn )
	&&(\alpha \in \T,\ \bftheta\in \T^d),
	\\
	\label{eq:parabhigh:TDef}
	T( \alpha, \theta ) &= \sum_{ n \in \Z } \omega(n) e( \alpha n^k + \theta n )
	&&(\alpha \in \T,\ \theta\in \T),
\end{align}
which may be viewed as Fourier transforms of
smoothed surface measures on 
$\{ (|\bfn|_k^k,\bfn) \,:\, \bfn \in [-2N,2N]^d \}$,
respectively for general $d$ and for $d = 1$.

Note that the sum over $\bfn \in \Z^d$ in~\eqref{eq:parabhigh:FDef} 
splits and we have
\begin{align}
\label{eq:parabhigh:Fsplitting}
	F( \alpha, \bftheta ) = \prod_{i=1}^d T( \alpha, \theta_i ).
\end{align}
Another useful observation is that
\begin{align}
\label{eq:parabhigh:FFourierSupport}
	\Supp(\wh{F}) \subset [-d(2N)^k,d(2N)^k] \times [-2N,2N]^d.
\end{align}

For each dyadic integer $Q$ and integer $s \geq 0$
such that $1 \leq Q \leq 2^s$, we define a piece
of our original exponential sum by
\begin{align}
\label{eq:parabhigh:FQsDef}
	F^{Q,s}(\alpha,\bftheta) 
	= \Phi_{Q,s}(\alpha) \cdot F(\alpha,\bftheta) . 
\end{align}
Recall that the weight $\Phi_{Q,s}$ is 
essentially a mollified indicator of the $\frac{1}{2^sN}$-neighborhood
of the set of rationals with denominator of size $Q$.

We now define the piece $F_\frakM$ of our exponential sum
corresponding to the union of all major arcs, 
and the piece $F_\frakm$ corresponding to the minor arcs, by
\begin{align}
\label{eq:parabhigh:FDcp}
	F_{\frakM} = \sum_{Q \leq N_1} \, \sum_{Q \leq 2^s \leq N} F^{Q,s},
	\qquad
	F_{\frakm} = F - F_{\frakM}.
\end{align}
Recalling the decomposition~\eqref{eq:mollif:rhoDef},
this means that
\begin{align}
\label{eq:parabhigh:FmDef}
	F_\frakm(\alpha,\bftheta)
	= \rho(\alpha) F(\alpha,\bftheta).
\end{align}

We fix a Weyl exponent $\tau = \max( 2^{1-k}, \frac{1}{k(k-1)})$.
The minor arc estimates of Appendix~\ref{sec:appweyl}
translate into the following statement.

\begin{proposition}
\label{thm:parabhigh:MinorArcBound}
Uniformly in $\alpha \in \T$, $\bftheta \in \T^d$, we have
\begin{align*}
	\rho(\alpha) \neq 0	
	\quad\Rightarrow\quad	
	|F(\alpha,\bftheta)| &\lesssim_\eps N^{d - d\tau + \eps}.
\end{align*}
\end{proposition}

\begin{proof}
Consider $\alpha \in \calU$ such that $\rho(\alpha) \neq 0$.
Take $1 \leq a \leq q \leq N^{k-1}$ such that $(a,q) = 1$
and $|\alpha - a/q| \leq 1/qN^{k-1}$.
It follows from Proposition~\ref{thm:mollif:rhoSupport} that
$q > N_1$, for else there exists a dyadic integer $Q$
such that $q \sim Q \Rightarrow Q \leq N_1$ and
$|\alpha - a/q| \leq 1/QN^{k-1}$, a contradiction.
Therefore we have $N \lesssim q \leq N^{k-1}$
and we may apply the bound of Proposition~\ref{thm:appweyl:MinorArcBound}
to each Weyl sum in the product~\eqref{eq:parabhigh:Fsplitting}.
\end{proof}

By~\eqref{eq:parabhigh:FmDef}, we have the following
immediate corollary.

\begin{corollary}
\label{thm:parabhigh:MinorArcPieceBound}
We have
\begin{align}
\label{eq:parabhigh:MinorArcPieceBound}
	\| F_{\frakm} \|_\infty 
	\lesssim_{\eps} N^{d - d\tau + \eps} .
\end{align}
\end{corollary}

We can derive a bound on the piece $F^{Q,s}$ of
the exponential sum by appealing to major arc bounds.

\begin{proposition}
\label{thm:parabhigh:ArcPiecePhysBound}
We have, uniformly for $Q \leq 2^s \leq N$,
\begin{align*}
\| F^{Q,s} \|_\infty
\lesssim_\eps
Q^\eps \bigg( \frac{2^s}{Q} \bigg)^{\frac{d}{k}} N^{d(1-\frac{1}{k})}.
\end{align*}
\end{proposition}

\begin{proof}
Consider $\alpha \in \calU$.
By~\eqref{eq:parabhigh:Fsplitting} and~\eqref{eq:parabhigh:FQsDef}, we have
\begin{align*}
	|F^{Q,s}(\alpha,\bftheta)|
	\leq
	\Phi_{Q,s}(\alpha) \prod_{j=1}^d |T(\alpha,\theta_i)|.
\end{align*}
If $\Phi_{Q,s}(\alpha) \neq 0$,
then it follows from~\eqref{eq:mollif:PhiSupport}
that there exist $1 \leq a \leq q$ with $(a,q) = 1$, $q \sim Q$
such that $|\alpha - \frac{a}{q}| \asymp \frac{1}{2^s N^{k-1}}$ if $2^s < \wt{N}$,
or $|\alpha - \frac{a}{q}| \lesssim \frac{1}{2^s N^{k-1}}$ if $2^s = \wt{N}$.
By~Proposition~\ref{thm:appweyl:MajorArcBound},
we have in both cases
\begin{align*}
|F^{Q,s}(\alpha,\bftheta)|
\lesssim_\eps
Q^{-\frac{d}{k} + \eps} (2^s N^{k-1} )^{\frac{d}{k}}. 
\end{align*}
\end{proof}

\begin{proposition}
\label{thm:parabhigh:ArcPieceFourierBound}
We have
\begin{align*}
\| \wh{F^{Q,s}} \|_\infty
\lesssim 
\frac{Q^2}{2^s N^{k-1}}.
\end{align*}
\end{proposition}

\begin{proof}
For any $(m,\bfell) \in \Z^{d+1}$, we have
\begin{align*}
	\wh{F^{Q,s}}(m,\bfell)
	&= \int_{\T^{d+1}} \Phi_{Q,s}(\alpha) F(\alpha,\bftheta) e( - \alpha m - \bftheta \cdot \bfell\, ) \dalpha \dbftheta
	\\
	&= \sum_{\bfn \in \Z^d} \omega_d(\bfn) 
	\int_{\T^{d+1}} \Phi_{Q,s}(\alpha) e\big( \alpha( |\bfn|_k^k - m ) + \bftheta \cdot ( \bfn - \bfell ) \big) \dalpha \dbftheta
	\\
	&= \omega_d(\bfell) \wh{\Phi}_{Q,s}( m - |\bfell|_k^k ).
\end{align*}
The result now follows from~\eqref{eq:mollif:PhiAverage}
and the trivial bound $\| \wh{\Phi}_{Q,s} \|_\infty \leq \| \Phi_{Q,s} \|_1$.
\end{proof}

From the previous physical and Fourier-side estimates
on a major arc piece $F^{Q,s}$,
we immediately deduce $L^1 \rightarrow L^\infty$
and $L^2 \rightarrow L^2$ estimates for the operator 
of convolution with this piece.

\begin{proposition}
\label{thm:parabhigh:ConvolEasyBounds}
Uniformly for $Q \leq 2^s \leq N$, we have
\begin{align}
\label{eq:parabhigh:L1LinftyBound}
\| F^{Q,s} \ast f \|_\infty 
&\lesssim_{\eps} Q^\eps
\Big( \frac{2^s}{Q} \Big)^{\frac{d}{k}} N^{d(1-\frac{1}{k})} \| f \|_1,
\\
\label{eq:parabhigh:L2L2EasyBound}
\| F^{Q,s} \ast f \|_2 
&\lesssim_\eps \frac{Q^2}{2^s N^{k-1}} \| f \|_2.
\end{align}
\end{proposition}

\begin{proof}
First note that for any bounded function $W : \T^{d+1} \rightarrow \C$, we have
\begin{align*}
	\| W \ast f \|_\infty \leq \| W \|_\infty \| f \|_1,
	\qquad
	\| W \ast f \|_2 = \| \wh{W} \wh{f} \|_2 \leq \| \wh{W} \|_\infty \| f \|_2.
\end{align*}
Applying these two inequalities to $W = F^{Q,s}$,
and inserting the estimates of Propositions~\ref{thm:parabhigh:ArcPiecePhysBound} 
and~\ref{thm:parabhigh:ArcPieceFourierBound},
we obtain the desired bounds.
\end{proof}

Interpolation between the previous convolution estimates 
gives the following result.

\begin{proposition}
\label{thm:parabhigh:InterpolEasyBound}
Let $p'_0 = \frac{2(k+d)}{d}$ and $p \in (1,2]$.
Uniformly for $Q \leq 2^s \leq N$, we have
\begin{align}
\label{eq:parabhigh:InterpolEasyGenBound}
\| F^{Q,s} \ast f \|_{p'}
\lesssim 
Q^{\frac{2}{p'} + \eps}  
\Big[ \Big( \frac{2^s}{Q} \Big)^\frac{d}{k} N^{d(1-\frac{1}{k})}\Big]^{1 - \frac{p'_0}{p'}} \| f \|_p.
\end{align}
\end{proposition}

\begin{proof}
Fix parameters $p \in (1,2]$ and $\theta \in (0,1]$ such that
\begin{align}
\label{eq:parabhigh:thetaDef}
\frac{1}{p'} = \frac{1-\theta}{\infty} + \frac{\theta}{2},
\qquad
\frac{1}{p} = \frac{1-\theta}{1} + \frac{\theta}{2}.
\end{align}
By interpolation
between the estimates of Proposition~\ref{thm:parabhigh:ConvolEasyBounds}, we obtain
\begin{align*}
\| F^{Q,s} \ast f \|_{p'}
&\lesssim 
Q^\eps
\Big( \frac{2^s}{Q} \Big)^{(1-\theta) \frac{d}{k} } N^{d(1-\frac{1}{k})(1-\theta)}
\cdot \Big( \frac{Q}{2^s} \Big)^\theta \Big( \frac{Q}{N^{k-1}} \Big)^\theta \cdot \| f \|_p
\\
&\lesssim 
Q^{\theta+\eps} \cdot \Big( \frac{2^s}{Q} \Big)^{ \frac{d}{k} - \frac{d}{k}(1+\frac{k}{d})\theta }
\cdot N^{d(1-\frac{1}{k}) - \theta(d(1-\frac{1}{k}) + k(1-\frac{1}{k}) ) } \cdot \| f \|_p
\\
&\lesssim 
Q^{\theta + \eps} \cdot \Big[ \Big( \frac{2^s}{Q} \Big)^{\frac{d}{k}} N^{ d(1-\frac{1}{k})} \Big]^{1 - \frac{k+d}{d} \theta } \cdot \| f \|_p.
\end{align*}
Since $\theta = \frac{2}{p'}$, we see
that $1 - \frac{k+d}{d} \theta = 1 - \frac{p'_0}{p'}$,
which yields the desired estimate.
\end{proof}

We need to sum this up over the major arcs.

\begin{proposition}
\label{thm:parabhigh:FinalConvolBound}
If $ p' > \frac{2(d+k) + 2k}{d} $, then 
\begin{align}
\label{eq:parabhigh:ConvolBound}
\| F_\frakM \ast f \|_{p'}
\lesssim 
N^{d - \frac{2(d+k)}{p'}}
\| f \|_p.
\end{align}
\end{proposition}

\begin{proof}
When $p' > p'_0$,
Proposition~\ref{thm:parabhigh:InterpolEasyBound} and
the triangle inequality yield
\begin{align*}
\| F_\frakM \ast f \|_{p'}
& \leq 
\sum_{Q \leq N} \sum_{Q \leq 2^s \leq N_1} \| F^{Q,s} \ast f \|_{p'} 
\\ 
& \lesssim 
\sum_{Q \leq N} \sum_{Q \leq 2^s \leq N_1} Q^{\frac{2}{p'} + \eps} 
\Big( \frac{2^s}{Q} \Big)^{ \frac{d}{k} (1 - \frac{p'_0}{p'}) } 
N^{ d(1-\frac{1}{k}) (1 - \frac{p'_0}{p'}) } \; \| f \|_{p}
\\ 
& \leq 
\sum_{Q \leq N} 
Q^{\frac{2}{p'} - \frac{d}{k} ( 1 - \frac{p'_0}{p'} ) + \eps}
N^{ d (1 - \frac{p'_0}{p'}) } \; \| f \|_{p} .
\end{align*}
The sum over the dyadic $Q$ is $O(1)$ for 
$(2 + \frac{d p'_0} {k}) \frac{1}{p'} < \frac{d}{k} $,
which gives the range stated in the proposition. 
\end{proof}

\textit{Proof of Theorem~\ref{thm:parabhigh:TruncRestrParab}.}
We have a decomposition $F = F_\frakM + F_\frakm$
which satisfies the estimates of Propositions~\ref{thm:parabhigh:MinorArcPieceBound}
and~\ref{thm:parabhigh:FinalConvolBound}.
The result now follows from Lemma~\ref{thm:prelims:TomasSteinDcp},
recalling that $\tau = \max( 2^{1-k}, \frac{1}{k(k-1)} )$.
\qed

\section{Restriction estimates for $k$-paraboloids of low dimension}
\label{sec:parablow}

In this section, we pursue the study of $k$-paraboloids of the form~\eqref{eq:intro:ParabSurface}
initiated in Section~\ref{sec:parabhigh}, but we aim at obtaining results
valid in the complete supercritical range of exponents $p > \frac{2(d + k)}{d}$ instead,
under a constraint on the dimension $d$.
The following is the main result of this section, 
which corresponds to Theorem~\ref{thm:intro:TruncRestrParabLowDim}.
Here $F_a$ is defined by~\eqref{eq:parabhigh:FaDef} as before.

\begin{theorem}
\label{thm:parablow:TruncRestrParab}
Suppose that $d \geq 1$, $k \geq 3$
and let $\tau = \max( 2^{1-k}, \frac{1}{k(k-1)})$.
Provided that $d < \frac{k^2-2k}{1 - k\tau}$,
for every $p > \frac{2(k+d)}{d}$ and $\eps > 0$, we have
\begin{align*}
	\int_{|F_a| \geq N^{d/2 - d\tau/2 + \eps} \|a\|_2 } |F_a|^p \dm
	\lesssim_{p,\eps} N^{\frac{dp}{2} - (k+d)} \| a \|_2^p.
\end{align*}
\end{theorem}

Note that lifting this result to a complete restriction estimate via 
Lemma~\ref{thm:prelims:SubcriticalCompletion}
would yield the same result as Corollary~\ref{thm:parabhigh:RestrParab}
with a more restrictive condition on $d$, 
therefore we do not carry out this process.
Our method of proof follows again the number-theoretic approach of Bourgain~\cite{Bourgain:ParabI}
for the parabola, this time in a fashion closer to the original.
Remarkably, this approach does not break down
when using the weaker minor arc estimates available for
the Weyl sums~\eqref{eq:parabhigh:TDef} associated to the $k$-paraboloid.
As in that reference, we first obtain a version of the desired estimate
which an extra factor $N^\eps$, whose proof is simpler and
serves as a blueprint for the more technical $\eps$-free case.
We fix at the outset a sequence $a : \Z^d \rightarrow \C$ supported on $[-N,N]^d$
with $\| a \|_2 = 1$, and we reuse the notation introduced in Section~\ref{sec:parabhigh}.
In particular we work again with the 
exponential sums~\eqref{eq:parabhigh:FDef} and~\eqref{eq:parabhigh:TDef},
and we fix again a Weyl exponent $\tau = \max( 2^{1-k}, \frac{1}{k(k-1)})$.

\subsection{Bounds on major and minor arc pieces of the exponential sum}

For each dyadic integer $Q$ and integer $s \geq 0$
such that $1 \leq Q \leq 2^s$, we define a piece
of our original exponential sum by
\begin{align}
\label{eq:parablow:FQsDef}
	F_{Q,s}(\alpha,\bftheta) 
	= F(\alpha,\bftheta) \Big[ \Phi_{Q,s}(\alpha) - \frac{\int \Phi_{Q,s}}{\int \rho} \rho(\alpha) \Big].
\end{align}
By comparison with the simpler definition~\eqref{eq:parabhigh:FQsDef},
the second term in the parenthesis
ensures that $F_{Q,s}$ satisfies good Fourier bounds at non-zero frequencies.
However, there is a trade-off in the sense that we only
get acceptable physical-side bounds on $F_{Q,s}$ for suffficiently small dimensions,
as the next proposition shows.

\begin{proposition}
\label{thm:parablow:ArcPiecePhysBound}
Suppose that $d < \frac{k^2 - 2k}{1 - k\tau}$. 
We have, uniformly for $Q \leq 2^s \leq N$,
\begin{align*}
	\| F_{Q,s} \|_\infty
	\lesssim_{\eps} \bigg( \frac{2^s}{Q} \bigg)^{\frac{d}{k}} Q^{\eps} N^{d(1-\frac{1}{k})}.
\end{align*}
\end{proposition}

\begin{proof}
From the definitions~\eqref{eq:parabhigh:FQsDef} and~\eqref{eq:parablow:FQsDef}, we have 
\begin{align*}
	F_{Q,s}(\alpha,\bftheta) 
	= F^{Q,s}(\alpha,\bftheta) 
	+ \frac{\int \Phi_{Q,s}}{\int \rho} \rho(\alpha) F(\alpha,\bftheta).
\end{align*}
By Propositions~\ref{thm:parabhigh:MinorArcBound}
and~\ref{thm:parabhigh:ArcPiecePhysBound},
and inserting the bounds~\eqref{eq:mollif:PhiAverage} 
and~\eqref{eq:mollif:rhoAverage},
we obtain
\begin{align*}
	|F_{Q,s}(\alpha,\bftheta)|
	\lesssim \Big( \frac{2^s}{Q} \Big)^{\frac{d}{k}} Q^{\eps}  N^{d - \frac{d}{k}}
	+ \frac{Q}{2^s} \cdot \frac{Q}{N} \cdot N^{d - (k - 2 + d\tau - \eps)}.
\end{align*}
Since $Q \leq 2^s \leq N$ and $(k-2)/(k^{-1}-\tau) > d$, 
the second term in the last line may be absorbed into the first
for $\eps$ small enough.
\end{proof}

In the rest of this section, we assume that 
the hypothesis $d < \frac{k^2 - 2k}{1 - k\tau}$ of
Theorem~\ref{thm:parablow:TruncRestrParab} is satisfied
to avoid repetition.
We also introduce a technical device analogous to that of Section~\ref{sec:powers} 
to ensure that all Fourier transforms under consideration stay inside 
an $N \times \dots \times N \times N^k$ box. 
We fix a trigonometric polynomial $\psi_N$ on $\T^{d+1}$
such that 
\begin{align*}
	[-d(2N)^k,d(2N)^k] \times [-2N,2N]^d  \prec \wh{\psi}_N 
	\prec [-2d(2N)^k,2d(2N)^k] \times [-4N,4N]^d,
\end{align*}
which in particular implies that $\int_{\T^{d+1}} \psi_N = 1$.
When $H : \T^{d+1} \rightarrow \C$ is a bounded measurable function, 
we write $\dot{H} = H \ast \psi_N$ for brevity; 
note that $\| \dot{H} \|_p \leq \| H \|_p$ for any $p \geq 1$ by Young's inequality, 
and that $F = \dot{F}$ by~\eqref{eq:parabhigh:FFourierSupport} and Fourier inversion.
With this notation in place, we derive a Fourier estimate
improving on that of Proposition~\ref{thm:parabhigh:ArcPieceFourierBound},
by exploiting the pseudorandomness of the weight 
$\Phi_{Q,s} - \frac{\int \Phi_{Q,s}}{\int \rho} \rho$.

\begin{proposition}
\label{thm:parablow:ArcPieceFourierBound}
Uniformly in $(m,\bfell) \in \Z^{d+1}$, we have
\begin{align*}
	|\wh{\dot{F_{Q,s}}}(m,\bfell)|
	&\lesssim_\eps 1_{|m| \lesssim N^k, |\bfell| \lesssim N}
	\Big( \frac{Q}{2^s N^{k-1}} 
	d( m - |\bfell|_k^k,2Q) + \frac{Q^2}{N^{2(k-1)-\eps}} \Big),
\end{align*}
In particular, we have
\begin{align*}
	\| \wh{\dot{F_{Q,s}}} \|_\infty
	&\lesssim_\eps \frac{Q}{2^s N^{k-1-\eps}}.
\end{align*}
\end{proposition}

\begin{proof}
Let $\Psi_{Q,s} = \Phi_{Q,s} - \frac{\int \Phi_{Q,s}}{\int \rho} \rho$
and note that $\wh{\Psi}_{Q,s}(0) = 0$.
By a computation similar to that in Proposition~\ref{thm:parabhigh:ArcPieceFourierBound},
we find that for any $(m,\bfell) \in \Z^{d+1}$,
\begin{align*}
	\wh{\dot{F}_{Q,s}}(m,\bfell)
	= \wh{\psi_N}(m,\bfell) \omega_d(\bfell) \wh{\Psi}_{Q,s}( |\bfell|_k^k - m ) 1_{ m \neq |\bfell|_k^k}.
\end{align*}
It then suffices to insert the estimates~\eqref{eq:mollif:PhiFourierBound}
as well as~\eqref{eq:mollif:PhiAverage},~\eqref{eq:mollif:rhoAverage} 
and~\eqref{eq:mollif:rhoFourierBound}.
\end{proof}

We again define a piece $F_\frakM$ of our exponential sum
corresponding to the union of all major arcs, 
and a piece $F_\frakm$ corresponding to the minor arcs, 
this time by
\begin{align}
\label{eq:parablow:FDcp}
	F_{\frakM} = \sum_{Q \leq N_1} \, \sum_{Q \leq 2^s \leq N} F_{Q,s},
	\qquad
	F_{\frakm} = F - F_{\frakM}.
\end{align}

\begin{proposition}
\label{thm:parablow:MinorArcPieceBound}
We have
\begin{align}
	\label{eq:parablow:MinorArcPieceBound}
	\| F_{\frakm} \|_\infty 
	\lesssim_{\eps} N^{d - d\tau + \eps}.
\end{align}
\end{proposition}

\begin{proof}
Recalling the definitions~\eqref{eq:parablow:FQsDef}
and~\eqref{eq:mollif:rhoDef}, we have
\begin{align*}
	F_{\frakm}(\alpha,\bftheta)
	&=
	F(\alpha,\bftheta) \bigg[ 1 - \sum_{Q \leq N_1} \sum_{Q \leq 2^s \leq N} 
	\Big( \Phi_{Q,s}(\alpha) - \frac{\int \Phi_{Q,s}}{\int \rho} \rho(\alpha) \Big) \bigg]
	\\
	&= \rho(\alpha) F(\alpha,\bftheta) \Bigg( 1 + \sum_{Q \leq N_1} \sum_{Q \leq 2^s \leq N} \frac{\int \Phi_{Q,s}}{\int \rho} \Bigg).
\end{align*}
From~\eqref{eq:mollif:PhiAverage} and~\eqref{eq:mollif:rhoAverage}, 
we deduce that
\begin{align*}
	|F_{\frakm}(\alpha,\bftheta)|
	&\lesssim \rho(\alpha) |F(\alpha,\bftheta)|
	\Bigg( 1 +  \sum_{Q \leq N_1} \, \sum_{Q \leq 2^s \leq N} \frac{Q^2}{2^s N^{k-1}} \Bigg)
	\\
	&\lesssim \rho(\alpha) |F(\alpha,\bftheta)|
	\Bigg( 1 +  \frac{1}{N^{k-1}} \sum_{Q \leq N'} Q \Bigg)	
	\\
	&\lesssim \rho(\alpha) |F(\alpha,\bftheta)|
\end{align*}
since $\sum_{Q \leq N'} Q \lesssim N' \leq N^{k-1}$.
It remains to insert the bound of 
Proposition~\ref{thm:parabhigh:MinorArcBound}
to conclude the proof.
\end{proof}

The previous estimates on $F_{Q,s}$
yield bounds for the operator of
convolution with this kernel.

\begin{proposition}
\label{thm:parablow:ConvolEasyBounds}
Uniformly for $Q \leq 2^s \leq N$, we have
\begin{align}
\label{eq:parablow:L1LinftyBound}
\| \dot{F}_{Q,s} \ast f \|_\infty 
&\lesssim_{\eps} \Big( \frac{2^s}{Q} \Big)^{\frac{d}{k}} Q^{\eps} N^{ d ( 1 - \frac{1}{k} ) } \| f \|_1,
\\
\label{eq:parablow:L2L2EasyBound}
\| \dot{F}_{Q,s} \ast f \|_2 
&\lesssim_\eps \frac{Q}{2^s N^{k-1-\eps}} \| f \|_2.
\end{align}
\end{proposition}

\begin{proof}
By the same argument as in Proposition~\ref{thm:parabhigh:ConvolEasyBounds},
inserting the estimates of Propositions~\ref{thm:parablow:ArcPiecePhysBound} 
and~\ref{thm:parablow:ArcPieceFourierBound} instead,
the proposition follows.
\end{proof}

Interpolation at the critical exponent
almost completely removes the operator constant,
as the next proposition shows.

\begin{proposition}
\label{thm:parablow:InterpolEasyBound}
Let $p'_0 = \frac{2(k+d)}{d}$.
Uniformly for $Q \leq 2^s \leq N$ and $p \in (1,2]$, we have
\begin{align}
\label{eq:parablow:InterpolEasyGenBound}
	\| \dot{F}_{Q,s} \ast f \|_{p'}
	\lesssim_\eps \Big[ \Big( \frac{2^s}{Q} \Big)^\frac{d}{k} N^{d(1-\frac{1}{k})}\Big]^{1 - \frac{p'_0}{p'}} N^\eps \| f \|_{p}
\end{align}
In particular, for $p' = p'_0$ we have
\begin{align}
\label{eq:parablow:InterpolEasyBound}
	\| \dot{F}_{Q,s} \ast f \|_{p'_0}
	\lesssim_\eps N^\eps \| f \|_{p_0}
\end{align}
\end{proposition}

\begin{proof}
Fix parameters $p \in (1,2]$ and $\theta \in (0,1]$ such that
\begin{align}
\label{eq:parablow:thetaDef}
	\frac{1}{p'} = \frac{1-\theta}{\infty} + \frac{\theta}{2},
	\qquad
	\frac{1}{p} = \frac{1-\theta}{1} + \frac{\theta}{2}.
\end{align}
By interpolation 
between the estimates of Proposition~\ref{thm:parablow:ConvolEasyBounds}, we obtain
\begin{align*}
	\| \dot{F}_{Q,s} \ast f \|_{p'}
	&\lesssim_\eps N^\eps \cdot \Big( \frac{2^s}{Q} \Big)^{(1-\theta) \frac{d}{k} } N^{d(1-\frac{1}{k})(1-\theta)}
	\cdot \Big( \frac{Q}{2^s} \Big)^\theta \Big( \frac{1}{N^{k-1}} \Big)^\theta \cdot \| f \|_p
	\\
	&\lesssim N^\eps \cdot \Big( \frac{2^s}{Q} \Big)^{ \frac{d}{k} - \frac{d}{k}(1+\frac{k}{d})\theta }
	\cdot N^{d(1-\frac{1}{k}) - \theta(d(1-\frac{1}{k}) + k(1-\frac{1}{k}) ) } \cdot \| f \|_p
	\\
	&\lesssim N^\eps \cdot \Big[ \Big( \frac{2^s}{Q} \Big)^{\frac{d}{k}} N^{ d(1-\frac{1}{k})} \Big]^{1 - \frac{k+d}{d} \theta } \cdot \| f \|_p.
\end{align*}
Since $\theta = \frac{2}{p'}$, we see
that $1 - \frac{k+d}{d} \theta = 1 - \frac{p'_0}{p'}$,
which yields the desired estimate.
\end{proof}

\subsection{$\eps$-full restriction estimates}

In this subsection we derive the upper bound
in Theorem~\ref{thm:parablow:TruncRestrParab} upto a factor $N^\eps$.
We fix a weight function $a : \Z^d \rightarrow \C$ 
supported in $[-N,N]^d$, and we may assume without loss of generality 
that $\| a \|_2 = 1$ in proving that variant of
Theorem~\ref{thm:parablow:TruncRestrParab}.
We introduce the usual level set $E_\lambda$ 
and weighted indicator $f$ defined by
\begin{align*}
	E_{\lambda} = \{ |F_a| \geq \lambda \},\qquad f = 1_{E_\lambda} \frac{F_a}{|F_a|}.
\end{align*}
Recall that the parameter $\lambda$ takes values in $(0,N^{d/2}]$.
The usual Tomas-Stein inequality~\eqref{eq:prelims:TomasStein}
(together with our earlier observation $F = \dot{F}$) becomes
\begin{align}
\label{eq:parablow:TomasSteinBound}
	\lambda^2 |E_\lambda|^2 \leq \langle \dot{F} \ast f , f \rangle.
\end{align}

\begin{proposition}
\label{thm:parablow:LevelSetBound}
Let $\eps > 0$ and $p'_0 =  \frac{2(k+d)}{d}$.
Uniformly for $\lambda \geq N^{d/2 - d\tau/2 + \eps}$,
we have
\begin{align*}
	|E_\lambda| \lesssim_\eps N^\eps \lambda^{-p'_0}.
\end{align*}
\end{proposition}

\begin{proof}
Starting from~\eqref{eq:parablow:TomasSteinBound}, 
and using the triangle and Hölder's inequalities, we obtain
\begin{align*}
	\lambda^2 |E_\lambda|^2 
	&\leq | \langle \dot{F}_{\frakM} \ast f , f \rangle |
	+ | \langle \dot{F}_{\frakm} \ast f , f \rangle |
	\phantom{\sum_{Q \leq N}}
	\\
	&\leq \sum_{Q \leq N'} \, \sum_{Q \leq 2^s \leq N} | \langle \dot{F}_{Q,s} \ast f , f \rangle |
	+ \| \dot{F}_\frakm \ast f \|_\infty \| f \|_1
	\\
	&\leq \sum_{Q \leq N'} \, \sum_{Q \leq 2^s \leq N}
	\| \dot{F}_{Q,s} \ast f \|_{p'_0} \| f \|_{p_0}
	+ \| F_{\frakm} \|_\infty \| f \|_1^2.
\end{align*}
By~\eqref{thm:parablow:MinorArcPieceBound}
and~\eqref{eq:parablow:InterpolEasyBound}, it follows that
\begin{align*}
	\lambda^2 |E_\lambda|^2 
	&\lesssim_\eps \sum_{Q \leq N'} \, \sum_{Q \leq 2^s \leq N} N^\eps \|f\|_{p_0}^2
	+ N^{d - d\tau + \eps} \| f \|_1^2
	\\
	&\lesssim_\eps N^\eps |E_\lambda|^{\frac{2}{p_0}} + N^{d - d\tau + \eps} |E_\lambda|^2.	
\end{align*}
Assuming that $\lambda \geq N^{d/2 - d\tau/2 + \eps}$, we infer that
\begin{align*}
	|E_\lambda|^{\frac{2}{p'_0}} \lesssim N^\eps \lambda^{-2}
	\quad\Rightarrow\quad
	|E_\lambda| \lesssim N^\eps \lambda^{-p'_0}.
\end{align*}
\end{proof}

The previous level set estimate
may be integrated into
a truncated $\eps$-full restriction estimate.

\begin{proposition}
\label{thm:parablow:epsFullRestrEst}
Let $\eps > 0$.
For $p \geq p'_0 = \frac{2(k+d)}{d}$, we have
\begin{align*}
	\int_{ |F_a| \geq N^{d/2 - d\tau/2 + \eps} } |F_a|^p \dm
	\lesssim_\eps N^{\frac{dp}{2} - (k+d) + \eps}.
\end{align*}
\end{proposition}

\begin{proof}
It suffices to invoke Proposition~\ref{thm:parablow:LevelSetBound} in
\begin{align*}
	\int_{|F_a| \geq N^{d/2 - d\tau/2 + \eps}} |F_a|^p \dm
	&= p \int_{N^{d/2 - d\tau/2 + \eps}}^{N^{d/2}} \lambda^{p-1} |E_\lambda| \dlambda
	\\
	&\lesssim_\eps N^\eps \int_1^{N^{d/2}} \lambda^{p - \frac{2(k+d)}{d} -1 } \dlambda
	\\
	&\lesssim_\eps N^{2\eps} \cdot N^{\frac{dp}{2} - (k+d)}.
\end{align*}
\end{proof}

\subsection{$\eps$-free restriction estimates}

The goal of this section is to derive Theorem~\ref{thm:parablow:TruncRestrParab} in full.
While we use propositions from the previous subsection, 
we do not need the final $\eps$-full estimate of Proposition~\ref{thm:parablow:epsFullRestrEst}.
We start by stating a distributional version
of Lemma~\ref{thm:powers:DivBound} 
(which follows immediately from Markov's inequality).

\begin{lemma}
\label{thm:parablow:DivBound2}
Let $D,Q,X \geq 1$ and $B \in \N$.
When $Q \leq 2X^{1/B}$, we have
\begin{align*}
	\#\{ |n| \leq X \,:\, d(n,Q) \geq D \} 
	\lesssim_{\eps,B} D^{-B} Q^\eps X.
\end{align*}
\end{lemma}

We tacitly assume that the letter $B$ 
denotes an integer from now on.
We may now establish a more precise version
of the estimate~\eqref{eq:parablow:L2L2EasyBound},
using divisor function bounds.

\begin{proposition}
\label{thm:parablow:ConvolHardBound}
Let $B, D \geq 1$. 
Uniformly for $Q \leq N^{k/B}$ and $Q \leq 2^s \leq N$,
\begin{align}
\label{eq:parablow:L2L2HardBound}
	\| \dot{F}_{Q,s} \ast f \|_2 
	\lesssim_{\eps,B}
	\frac{Q^{1+\eps}}{2^s N^{k-1}}
	\big( D \| f \|_2 + D^{-\frac{B}{2}} N^{\frac{k+d}{2}} \| f \|_1 \big).
\end{align}
\end{proposition}

\begin{proof}
Note that $I \coloneqq \| \dot{F}_{Q,s} \ast f \|_2 = \| \wh{\dot{F}}_{Q,s} \wh{f} \|_2$.
Via the bounds of Proposition~\ref{thm:parablow:ArcPieceFourierBound}, we obtain
\begin{align*}
	I 
	&= \Bigg[ \sum_{\substack{ |m| \lesssim N^k \\ |\bfell| \lesssim N }} 
	|\wh{\dot{F}}_{Q,s}(m,\bfell)|^2 |\wh{f}(m,\bfell)|^2 \Bigg]^{1/2}
	\\
	&\lesssim \frac{Q}{2^s N^{k-1}} 
	\Bigg[ \sum_{\substack{ |m| \lesssim N^k \\ |\bfell| \lesssim N }} 
	d( m - |\bfell|_k^k,2Q)^2 |\wh{f}(m,\bfell)|^2 \Bigg]^{1/2}
	+ \frac{Q^2}{2^s N^{2(k-1) - \eps}} \| \wh{f} \|_2
\end{align*}
Writing $n = m - |\bfell|_k^k$, assuming $Q \leq N^{k/B}$ and invoking 
Lemma~\ref{thm:parablow:DivBound2}, we obtain
\begin{align*}
	I
	&\lesssim_{\eps,B} \frac{Q}{2^s N^{k-1}}
	\bigg[ D^2 \| \wh{f} \|_2^2 + \| \wh{f} \|_\infty^2 N^d \times \#\{ |n| \lesssim N^k \,:\, d(n,2Q) > D \} \bigg]^{1/2}
	+ \frac{Q^2}{2^s N^{2(k-1)-\eps}} \| f \|_2
	\\
	&\lesssim \frac{Q}{2^s N^{k-1}} \Big( D^2 \| f \|_2^2 + D^{-B} Q^\eps N^{k+d} \| f \|_1^2 \Big)^{1/2}
	+ \frac{Q}{2^s N^{k-1}} \cdot \frac{Q}{2^s N^{k-1-\eps}} \| f \|_2.
\end{align*}
Since $Q \leq 2^s$, the last term may be absorbed into the first.
Finally we obtain
\begin{align*}
	I \lesssim \frac{Q^{1 + \eps}}{2^s N^{k-1}} 
	\big( D \| f \|_2 + D^{-\frac{B}{2}} N^{\frac{k+d}{2}} \| f \|_1 \big).
\end{align*}
\end{proof}

With this more precise $L^1 + L^2 \rightarrow L^2$ estimate in hand,
we proceed to interpolate with the $L^1 \rightarrow L^\infty$
estimate as before.

\begin{proposition}
\label{thm:parablow:InterpolHardBound}
Let $B,D \geq 1$.
Let $p'_0 = \frac{2(k+d)}{d}$ and $p' \in (2,\infty)$.
Uniformly for $Q \leq N^{k/B}$ and $Q \leq 2^s \leq N$, we have
\begin{align*}
	\| \dot{F} \ast f \|_{p'}
	\lesssim_{\eps,B}
	Q^\eps \Big[ \Big( \frac{2^s}{Q} \Big)^{\frac{d}{k}} N^{d(1-\frac{1}{k})} \Big]^{1 - \frac{p'_0}{p'}}
	\big( D^{\frac{2}{p'}}	\| f \|_p + D^{-\frac{B}{p'}} N^{\frac{k+d}{p'}} \| f \|_1 \big).
\end{align*}
\end{proposition}

\begin{proof}
Consider the real number $\theta \in (0,1)$ such that~\eqref{eq:parablow:thetaDef} holds.
By convexity of $L^p$ norms, we have
\begin{align*}
	I \coloneqq \| \dot{F}_{Q,s} \ast f \|_{p'} 
	\leq \| \dot{F}_{Q,s} \ast f \|_{\infty}^{1 - \theta} \| \dot{F}_{Q,s} \ast f \|_2^\theta.
\end{align*}
Applying~\eqref{eq:parablow:L1LinftyBound},
and~\eqref{eq:parablow:L2L2HardBound},
we obtain
\begin{align*}
	I &\lesssim_{\eps,B}
	Q^\eps \cdot \Big( \frac{2^s}{Q} \Big)^{(1-\theta)\frac{d}{k}} N^{(1-\theta)d(1-\frac{1}{k})} 
	\cdot \Big( \frac{Q}{2^s} \Big)^\theta \Big( \frac{1}{N^{k-1}} \Big)^\theta 
	\\
	&\phantom{\lesssim_{\eps,B} .} 
	\times \big( D^\theta \| f \|_1^{1-\theta} \| f \|_2^\theta 
	+ D^{-\theta \frac{B}{2}} N^{\theta \frac{k+d}{2}} \cdot \| f\|_1 )
\end{align*}
Since $|f|$ takes values in $\{0,1\}$, we may rewrite this as
\begin{align*}
	I &\lesssim_{\eps,B} 
	Q^\eps \Big[ \Big( \frac{2^s}{Q} \Big)^{\frac{d}{k}} N^{ d(1-\frac{1}{k}) } \Big]^{1 - \theta \frac{k+d}{d}} 
	\big( D^\theta \| f \|_p +  D^{ -\theta \frac{B}{2} } N^{ \theta \frac{k+d}{2} } \| f \|_1 \big).
\end{align*}
The proof is finished upon observing that 
$\theta = \frac{2}{p'}$ by~\eqref{eq:parablow:thetaDef},
and recalling that $p'_0 = \frac{2(k+d)}{d}$.
\end{proof}

Following the argument of Bourgain~\cite{Bourgain:ParabI},
we distinguish two cases according to the size of $Q$.
We introduce a parameter $Q_1 \geq 1$, and we write
$F_\frakM = F_1 + F_2$ with
\begin{align}
\label{eq:parablow:F1F2Def}
	F_1 = \sum_{Q \leq Q_1} \sum_{Q \leq 2^s \leq N} F_{Q,s},
	\qquad
	F_2 = \sum_{Q_1 < Q \leq N_1} \sum_{Q \leq 2^s \leq N} F_{Q,s}.
\end{align}

\begin{proposition}
\label{thm:parablow:F1Bound}
Suppose that $p' > p'_0$. 
Let $T \geq 1$ and suppose that $1 \leq Q_1 \leq N^{k/B}$.
Then
\begin{align*}
	\| \dot{F}_1 \ast f \|_{p'}
	\lesssim
	N^{ d (1 - \frac{p'_0}{p'}) }
	\big( T^2 \| f \|_p
	+ T^{-B} N^{\frac{k+d}{p'}} \| f \|_1 \big).
\end{align*}
\end{proposition}

\begin{proof}
By the triangle inequality
and Proposition~\ref{thm:parablow:InterpolHardBound} 
with $T = D^{1/p'}$, it follows that
\begin{align*}
	\| \dot{F}_1 \ast f \|_{p'}
	&\lesssim \sum_{Q \leq Q_1} Q^{\eps - \frac{d}{k} ( 1 - \frac{p'_0}{p'} ) }
	\sum_{2^s \leq N} (2^s)^{ \frac{d}{k} (1-\frac{p'_0}{p'}) }  N^{ d ( 1 -\frac{1}{k}) (1 - \frac{p'_0}{p'}) } 
	\\
	&\phantom{\lesssim .}
	\cdot \big( T^2 \| f \|_p + T^{- B} N^{\frac{k+d}{p'}} \| f \|_1 \big).
	\\
	&\lesssim N^{ d (1 - \frac{p'_0}{p'}) }
	\big( T^2 \| f \|_p + T^{- B} N^{\frac{k+d}{p'}} \| f \|_1 \big).
\end{align*}
\end{proof}

We now consider the piece $F_2$ involving
large values of the parameter $Q$.

\begin{proposition}
\label{thm:parablow:F2Bound}
Let $p' > p'_0$.
We have
\begin{align*}
	\| \dot{F}_2 \ast f \|_{p'}
	\lesssim N^\eps Q_1^{ - \frac{d}{k} (1 - \frac{p'_0}{p'}) } 
	N^{ d (1-\frac{p'_0}{p'}) } \| f \|_p.
\end{align*}
\end{proposition}

\begin{proof}
From the triangle inequality and~\eqref{eq:parablow:InterpolEasyGenBound},
we deduce that
\begin{align*}
	\| \dot{F}_2 \ast f \|_{p'} 
	&\lesssim 
	\sum_{Q > Q_1} Q^{- \frac{d}{k} (1 - \frac{p'_0}{p'}) }	
	\sum_{2^s \leq N} (2^s)^{ \frac{d}{k} (1 - \frac{p'_0}{p'}) } 
	\cdot N^\eps N^{ d (1-\frac{1}{k}) (1-\frac{p'_0}{p'} ) }
	\cdot \| f \|_p
	\\
	&\lesssim N^\eps Q_1^{ - \frac{d}{k} (1 - \frac{p'_0}{p'}) } 
	N^{ d (1-\frac{p'_0}{p'}) } \| f \|_p.
\end{align*}
\end{proof}

\begin{proposition}
\label{thm:parablow:epsFreeLevelSetEst}
For $\frac{2(k+d)}{d} < q \lesssim 1$,
\begin{align*}
	|E_\lambda| 
	\lesssim_{\eps,q}
	N^{\frac{dq}{2} - (k+d)} \lambda^{-q}
	\qquad
	\text{for $\lambda \geq N^{d/2 - d\tau/2 + \eps}$}.	
\end{align*}
\end{proposition}

\begin{proof}
Starting from~\eqref{eq:parablow:TomasSteinBound},
and recalling the decompositions~\eqref{eq:parablow:FDcp} 
and~\eqref{eq:parablow:F1F2Def}, we have,
for any $p' > p'_0$,
\begin{align*}
	\lambda^2 |E_\lambda|^2
	&\leq |\langle \dot{F}_{\frakm} \ast f , f \rangle|
	+ |\langle \dot{F}_2 \ast f , f \rangle|
	+ |\langle \dot{F}_1 \ast f , f \rangle|
	\\
	&\leq \| F_\frakm \|_\infty \| f \|_1^2 +
	\| \dot{F}_2 \ast f \|_{p'} \| f \|_p
	+ \| \dot{F}_1 \ast f \|_{p'} \| f \|_p.
\end{align*}
Let $T \geq 1$ be a parameter to be determined later,
and assume that we have chosen $Q_1$
so that $Q_1 \leq N^{k/B}$.
Inserting the estimates of
Propositions~\ref{thm:parablow:MinorArcPieceBound},~\ref{thm:parablow:F1Bound},
and~\ref{thm:parablow:F2Bound}, this yields
\begin{align*}
	\lambda^2 |E_\lambda|^2
	&\lesssim N^{d - d\tau + \eps} |E_\lambda|^2
	+ N^\eps Q_1^{ - \frac{d}{k} (1-\frac{p'_0}{p'}) } N^{ d (1 - \frac{p'_0}{p'}) } \| f \|_p^2
	\\
	&\phantom{\lesssim .} 
	+ T^2	N^{ d (1 - \frac{p'_0}{p'}) } \| f \|_p^2 
	+ T^{-B} 	N^{ d (1 - \frac{p'_0}{p'}) + \frac{k+d}{p'} } \| f \|_p \| f \|_1.
\end{align*} 
Assume that $\lambda \geq N^{d/2 - d\tau/2 + \eps}$
and fix $Q_1 = N^{\eps_1}$, where $\eps_1 = k/2B$.
Provided that $\eps$ is small enough, we have then
\begin{align*}
	\lambda^2 |E_\lambda|^2 
	\lesssim T^2 N^{ d - \frac{2(k+d)}{p'} } |E_\lambda|^{2 - \frac{2}{p'}}
	+ T^{-B} N^{ d - \frac{(k+d)}{p'} } |E_\lambda|^{2 - \frac{1}{p'}}.
\end{align*}
Writing $\lambda = \eta N^{d/2}$ with $\eta \in (0,1]$,
we have either
\begin{align*}
	|E_\lambda|^{\frac{2}{p'}}
	\lesssim T^2 N^{ - \frac{2(k+d)}{p'}} \eta^{-2}
	\quad\text{or}\quad
	|E_\lambda|^{\frac{1}{p'}}
	\lesssim T^{-B} N^{ - \frac{k+d}{p'}} \eta^{-2}.
\end{align*}
Write $D = T^{p'}$, so that in either case
\begin{align*}
	|E_\lambda|
	\lesssim D N^{ - (k+d)} \eta^{-p'}
	+ D^{-B} N^{- (k+d) } \eta^{-2p'}.
\end{align*}
Choose $D = \eta^{-\nu}$
for parameter $\nu> 0$, so that
\begin{align*}
	|E_{\lambda}| \lesssim N^{- (k+d) } \eta^{- p' - \nu } ( 1 + \eta^{ - p' + (B+1) \nu } ).
\end{align*}
Choosing $B \geq C/\nu$ with $C > 0$ large enough,
we deduce that $|E_\lambda| \lesssim N^{-(k+d)} \eta^{- p' - \nu}$.
Since $q \coloneqq p' + \nu$ can be chosen arbitrarily close to $\frac{2(k+d)}{d}$,
this concludes the proof, upon recalling that $\eta = \lambda N^{-d/2}$.
\end{proof}

\textit{Proof of Theorem~\ref{thm:parablow:TruncRestrParab}.}
We apply Proposition~\ref{thm:parablow:epsFreeLevelSetEst} 
for a certain $\frac{2(k+d)}{d} < q < p$ to obtain
\begin{align*}
	\int_{|F_a| \geq N^{d/2 - d\tau/2 + \eps}}
	|F_a|^p \dm
	&= p \int_{N^{d/2 - d\tau/2 + \eps}}^{N^{d/2}}
	\lambda^{p-1} |E_\lambda| \dlambda
	\\
	&\lesssim_{p,\eps}
	N^{\frac{dq}{2} - (k+d)} \int_1^{N^{d/2}} \lambda^{p - q - 1} \dlambda.
	\\
	&\lesssim_p
	N^{\frac{dp}{2} - (k+d)}.
\end{align*}
\qed

\appendix

\section{Bounds on Weyl sums}
\label{sec:appweyl}

We fix an integer $k \geq 2$.
Recall that we defined the Weyl sum $T$ by~\eqref{eq:parabhigh:TDef}.
In our argument, we make use several times of
the following standard minor arc bound.

\begin{proposition}
\label{thm:appweyl:MinorArcBound}
Let $\tau = \min( 2^{1-k} , \frac{1}{k(k-1)})$.
Suppose that $\alpha \in \T$, $1 \leq a \leq q$ are such that
$|\beta| = \| \alpha - \frac{a}{q} \| \leq \frac{1}{q^2}$ 
and $N \lesssim q \lesssim N^{k-1}$.
For every $\eps > 0$, we have
\begin{align*}
	|T( \alpha,\theta )| 
	\lesssim_\eps
	N^{1 - \tau + \eps}
\end{align*}
\end{proposition}

\begin{proof}
When $\tau = 2^{1-k}$,
this is a consequence of Weyl's inequality~\cite[Lemma~2.4]{Vaughan:Book}
(the presence of a smooth weight does not affect 
the squaring-differencing argument significantly).
We let $J_{s,k}(N)$ denote the number of solutions $n_1,\dots,n_s,m_1,\dots,m_s \in [N]$
to the system
\begin{align*}
	&\phantom{(1 \leq j \leq k)} &
	n_1^j + \dots + n_s^j &= m_1^j + \dotsb m_s^j
	&&(1 \leq j \leq k).
\end{align*}
The Vinogradov method~\cite[Theorem~5.2]{Vaughan:Book} gives the bound
\begin{align*}
	|T(\alpha,\theta)|
	\lesssim \Big[ (q^{-1} + N^{-k} + q N^{-k}) N^{\frac{1}{2}k(k-1)} J_{s,k-1}(N) \Big]^{\frac{1}{2s}} \log N,
\end{align*}
since the weight $\omega$ is eliminated in the application of the
multidimensional sieve~\cite[Chapter~5]{Vaughan:Book}.  
The latest bound on the Vinogradov mean value~\cite{BDG:VinoMeanValue}
gives $J_{s,k-1}(N) \lesssim_\eps N^{2s - \frac{1}{2}k(k-1) + \eps}$ for $s = \frac{1}{2} k(k-1)$.
Under our assumptions on $q$, it follows that
$|T(\alpha,\theta)| \lesssim_\eps N^{1-\frac{1}{k(k-1)} + \eps}$.
\end{proof}

On the major arcs, we use a majorant obtained 
through the Poisson formula and standard bounds
on oscillatory integrals and Gaussian sums.

\begin{proposition}
\label{thm:appweyl:MajorArcBound}
Let $k \geq 3$.
Suppose that $|\beta| = \| \alpha - a/q \| \lesssim 1/qN^{k-1}$, 
$1 \leq a \leq q \lesssim N$, $(a,q) = 1$.
For every $\eps > 0$, we have
\begin{align*}
	|T(\alpha,\theta)| 
	\lesssim_\eps
	q^{-1/k + \eps} \min(N,|\beta|^{-1/k}).
\end{align*}
\end{proposition}

\begin{proof}
Recall that we chose a weight of the form 
$\omega = \eta( \frac{\cdot}{N} )$,
where $\eta$ is supported on $[-2,2]$.
We define a Gaussian sum and an oscillatory integral by
\begin{align}
\label{eq:appweyl:GaussianSumOscIntg}
	S(a,b;q)
	= \sum_{u \bmod q} e_q( a u^k + b u ),
	\qquad
	J( \beta,\gamma ; N )
	= \int_{\R} \eta(x) e( \beta N^k x^k + \gamma N x ) \dx.	
\end{align}
Recalling~\eqref{eq:parabhigh:TDef}, writing
$\alpha \equiv \frac{a}{q} + \beta \bmod 1$
and summing over residue classes modulo $q$, we obtain
\begin{align*}
	T(\alpha,\theta)
	= \sum_{u \bmod q} e_q( a u^k ) 
	\sum_{\substack{n \in \Z \,: \\ n \equiv u \bmod q }} \omega(n) e( \beta n^k + \theta n ).
\end{align*}
Writing $1_{n \equiv u \bmod q} = q^{-1} \sum_{b \bmod q} e_q(b(u-n))$, we arrive at
\begin{align*}
	T( \alpha,\theta )
	= \sum_{b \bmod q} q^{-1} S(a,b;q)
	\sum_{n \in \Z} \omega(n) e( \beta n^k + ( \theta - \tfrac{b}{q} ) n ).  
\end{align*}
By Poisson's formula and a change of variable, we deduce that
\begin{align}
	T( \alpha,\theta )
	\label{eq:appweyl:PoissonDcp}
	= \sum_{b \bmod q} q^{-1} S(a,b;q) 
	\sum_{m \in \Z} N J( \beta, \theta - \tfrac{b}{q} - m ; N ).
\end{align}
We write $J( \beta, \theta - \tfrac{b}{q} - m ; N ) = \int_\R \eta(x) e( N \phi_{b,m}(x) ) \dx$,
where
\begin{align*}
	\phi_{b,m}(x) 
	= \beta N^{k-1} x^k + ( \theta - \tfrac{b}{q} - m ) x.
\end{align*}
On the support of $\eta$, we have $|x| \leq 2$ and therefore
\begin{align*}
	\phi_{b,m}'(x)
	= \theta - \tfrac{b}{q} - m + O( \tfrac{1}{q} )
\end{align*}
under our size condition on $\beta$.
We fix a large enough constant $C>0$.

For $|m| \geq C$, we have $|\phi'_{b,m}| \asymp |m|$ on $\Supp \eta$,
and therefore by stationary phase~\cite[Chapter VII]{Stein:Book}
we have $| \int_\R \eta  e( N \phi_{b,m}  ) |  \lesssim (N|m|)^{-2}$.

For $\| \theta - \frac{b}{q} \| \geq \frac{C}{q}$,
we have $|\phi_{b,m}'| \asymp |\theta - \frac{b}{q} - m| \gtrsim \| \theta - \frac{b}{q} \|$
on $\Supp \eta$ and $\| \frac{\phi_{b,m}}{|\theta - \frac{b}{q} - m|} \|_{C^2} \lesssim 1$,
so that by stationary phase again we deduce that
$| \int_\R \eta  e( N \phi_{b,m}  ) | \lesssim (N \| \theta - \frac{b}{q} \|)^{-1}$.

Finally, for $|m| \leq C$ and $ \| \theta - \frac{b}{q} \| \leq \frac{C}{q}$,
we observe that $|N\phi_{b,m}^{(k)}| \asymp_k |\beta| N^k$ on $\R$,
so that by a basic van der Corput estimate~\cite[Chapter VII]{Stein:Book},
we obtain $| \int_\R \eta  e( N \phi_{b,m}  ) | \lesssim (1 + |\beta|N^k)^{-1/k}$.
For the Gaussian sum, we use a classical bound of Hua~\cite[Theorem~7.1]{Vaughan:Book}:
$|q^{-1} S(a,b;q)| \lesssim_\eps q^{-\frac{1}{k} + \eps}$ for $(a,q) = 1$.
Inserting these various estimates into~\eqref{eq:appweyl:PoissonDcp} yields
\begin{align*}
	|T(\alpha,\theta)|
	&\lesssim_\eps
	q^{-1/k + \eps} \sum_{\substack{ \| \theta - \frac{b}{q} \| \leq \frac{C}{q} \\ |m| \leq C }} N (1 + |\beta| N^k)^{-\frac{1}{k}}
	\\
	&\phantom{\lesssim_\eps .}
	+ q^{-1/k + \eps} \sum_{\substack{ \| \theta - \frac{b}{q} \| \geq \frac{C}{q} \\ |m| \leq C }} \| \theta - \tfrac{b}{q} \|^{-1}
	\ +\ q^{1 - 1/k +\eps} \sum_{|m| \geq C} N^{-1} |m|^{-2}
	\\
	&\lesssim 
	q^{-1/k + \eps} N (1 + |\beta| N^k)^{-\frac{1}{k}} \,+\, q^{1 - 1/k + \eps}.
\end{align*}
The second term may be absorbed into the first since
$|\beta| \lesssim \frac{1}{qN^{k-1}}$
and $1 \leq q \lesssim N$, and this concludes the proof.
\end{proof}

\bibliographystyle{amsplain}
\bibliography{epsremoval_parab}

\bigskip

\textsc{\footnotesize Department of mathematics,
University of British Columbia,
Room 121, 1984 Mathematics Road,
Vancouver BC V6T 1Z2, Canada
}

\textit{\small Email address: }\texttt{\small khenriot@math.ubc.ca}

\textsc{\footnotesize Heilbronn Institute for Mathematical Research
School of Mathematics,
University of Bristol,
Howard House, Queens Avenue, 
Bristol BS8 1SN, United Kingdom
}

\textit{\small Email address: }\texttt{\small kevin.hughes@bristol.ac.uk}

\end{document}